\documentclass{article}
\usepackage{graphicx} % Required for inserting images
\usepackage{amsmath, amsthm, a4, latexsym, amssymb, mathtools, enumitem, xcolor, color, tikz-cd,todonotes}
\usepackage[unicode]{hyperref}
\usepackage[normalem]{ulem} % For strikethrough
\hypersetup{
        colorlinks,			% farbl. hervorhebung von links (-> keine Kästen)
        citecolor=blue,		% default rot
        urlcolor=blue,
        linkcolor=blue
}

\usepackage{mlmodern} %for nice font weight
\usepackage[T1]{fontenc} %for accents

\newtheorem{theorem}{Theorem}
\newtheorem{proposition}[theorem]{Proposition}

\newtheorem{lemma}[theorem]{Lemma}
\newtheorem{claim}[theorem]{Claim}

\theoremstyle{definition}
\newtheorem{definition}[theorem]{Definition}

\newtheorem{remark}[theorem]{Remark}

\newtheorem{question}[theorem]{Question}

\DeclareMathOperator{\aut}{Aut}
\DeclareMathOperator{\SO}{SO}
\DeclareMathOperator{\SU}{SU}
\DeclareMathOperator{\id}{id}

\newcommand{\Fraczyk}{Fr\k{a}czyk}

\newcommand{\ts}{\textsuperscript}

\title{Vanishing uniqueness thresholds in Voronoi percolation on products}

\author{
Matteo \textsc{D'ACHILLE}\thanks{Institut Élie Cartan de Lorraine, CNRS, Universit\'e de Lorraine, F-57070, Metz, France\newline $_{}$\hfill  \href{mailto:matteo.d-achille@univ-lorraine.fr}{\texttt{matteo.d-achille@univ-lorraine.fr}}}, Jan \textsc{GREBIK}\thanks{Universität Leipzig, Mathematisches Institut, D-04009, Leipzig, Germany\newline $_{}$\hfill  \href{mailto:grebikj@gmail.com}{\texttt{grebikj@gmail.com}}},
Ali \textsc{KHEZELI}\thanks{School of Mathematics, Institute for Research in Fundamental Sciences, Tehran, Iran\newline $_{}$\hfill  \href{mailto:alikhezeli@ipm.ir}{\texttt{alikhezeli@ipm.ir}}},\\
Konstantin \textsc{RECKE}\thanks{Mathematical Institute, University of Oxford, UK \newline $_{}$\hfill  \href{mailto:konstantin.recke@maths.ox.ac.uk}{\texttt{konstantin.recke@maths.ox.ac.uk}}}, and Amanda \textsc{WILKENS}\thanks{Department of Mathematical Sciences, Carnegie Mellon University, Pittsburgh, USA \newline $_{}$\hfill  \href{mailto:awilkens.math@gmail.com}{\texttt{awilkens.math@gmail.com}}} }

\date{}

\begin{document}

\maketitle

\begin{abstract}
    We study Poisson--Voronoi percolation and its discrete analogue Bernoulli--Voronoi percolation in spaces with a non-amenable product structure. We develop a new method of proving smallness of the uniqueness threshold $p_u(\lambda)$ at small intensities $\lambda>0$ based on the unbounded borders phenomenon of their underlining ideal Poisson--Voronoi tessellation. We apply our method to several concrete examples in both the discrete and the continuum setting, including $k$-fold graph products of $d$-regular trees for $k\ge2,d\ge3$ and products of hyperbolic spaces  $\mathbb H_{d_1}\times \ldots \times \mathbb H_{d_k}$ for $k\ge2, d_i\ge2$, complementing a recent result of the 2\ts{nd} and 4\ts{th} author for symmetric spaces of connected higher rank semisimple real Lie groups with property~(T). We also provide new examples of non-amenable Cayley graphs with the FIID sparse unique infinite cluster property, answering positively a recent question of Pete and Rokob.
\end{abstract}

\section{Introduction}\label{sec:intro}

{\bf Poisson-Voronoi percolation} is a continuum percolation model that can be defined on any metric space $(M,d)$ with an infinite Radon measure $\mu$ in two steps.
First, for $\lambda>0$, consider a Poisson point process of intensity $\lambda\mu$ and associate to each point of the process its Voronoi cell, i.e.~the set of all points in $M$ closer to this point than to any other point of the process. For $p\in(0,1]$, color each cell black with probability $p$ and white with probability $1-p$ independently of the colors of all other cells and let $\omega_p^{(\lambda)}$ denote the union of black cells. The discrete analogue is {\bf Bernoulli--Voronoi percolation} on a locally finite, connected graph. Here, for $\lambda\in(0,1]$, we first perform $\lambda$-Bernoulli site percolation, i.e.~delete vertices independently with probability $1-\lambda$, and associate to each random point its Voronoi cell. Then we independently color each cell black with probability $p\in(0,1]$ and consider the union of black cells, also denoted by~$\omega_p^{(\lambda)}$.

The Poisson--Voronoi percolation model was first studied in Euclidean space $\mathbb R^d$, see~e.g.~\cite{Z96,BR06,BR06b}. Then Benjamini and Schramm showed in their seminal paper~\cite{BS00} that fascinating new behavior arises in the negatively curved setting of the hyperbolic plane $\mathbb H_2$; we refer the interested reader to~\cite{HM24} for recent results in this setting and a state-of-the-art overview. In this spirit, the study of Poisson--Voronoi percolation on more general manifolds, which goes back already to \cite{BS98}, as well as Bernoulli--Voronoi percolation on graphs beyond $\mathbb Z^d$ has recently attracted attention \cite{BDRS25,GR25}. 
The results in \cite{GR25}, by the second and fourth authors, partly motivated this paper, so we briefly describe them below before stating the main results of this paper.

For fixed $\lambda$, a fundamental parameter associated to Poisson/Bernoulli--Voronoi percolation is the {\bf uniqueness threshold}
$$
p_u(\lambda)=\inf \big\{ p \, : \, \omega^{(\lambda)}_{p} \, \, \text{has a unique unbounded cluster w.p.p.} \big\},
$$
where {\bf cluster} refers to a path-connected component of $\omega_p^{(\lambda)}$ and ``w.p.p.'' stands for ``with positive probability''. The main result of~\cite{GR25} gives examples of Riemannian symmetric spaces of connected semisimple real Lie groups such that 
$$
\lim_{\lambda \to 0} p_u(\lambda)=0.
$$
This behavior is in striking contrast to the results of Benjamini and Schramm in the hyperbolic plane, where $\lim_{\lambda\to0} p_u(\lambda)=1$. In our opinion, it might be of interest for two more reasons:
\begin{itemize}
\item It has applications to ergodic theory, where it can be used to produce examples of non-amenable groups with {\em FIID sparse unique infinite clusters}. We describe these applications in more detail in Section~\ref{sec:FSUICP}.
\item To the best of our knowledge, there is no general methodology to prove the smallness of $p_u$. 
\end{itemize} 
To elaborate on the second point, smallness of $p_u$ was proved in~\cite{GR25} in the case that the associated Lie group has {\em higher rank} and {\em property~(T)}. The higher rank--assumption entered via a recent breakthrough of \Fraczyk, Mellick, and the fifth author~\cite{FMW23}.
%about low intensity limits of Poisson--Voronoi tessellations. 
Property~(T) entered via a probabilistic consequence first developed in the discrete setting in~\cite{MR23}. However, there are some very interesting examples of higher-rank symmetric spaces such that the associated Lie group does {\em not} have property~(T), for instance the Riemannian product $\mathbb H_2\times\mathbb H_2$ of two hyperbolic planes. See~\cite{H01} for the definition of Riemannian product and note that as a metric space it corresponds to the $L^2$-product. Note that there are also discrete examples with ``higher rank''--like properties but no underlying property~(T) such as the direct product of two $3$-regular trees.

\smallskip
In this paper, we develop a new method for proving smallness of $p_u$ which applies in many situations where property~(T) cannot be used. Before describing this method, we state our main results.
%\sout{its most striking applications} \newMD{some possibly surprising applications}. {\color{blue} I am not sure about using surprising because we expected this, but it was not clear how to prove. Perhaps, what about: ``its main applications" or ``its most emblematic applications"?} \newMD{The ``possibly surprising'' is assuming the point of view of a reader which might (probably) be not in the field (by the way, all my other style remarks are guided by this principle. Of course it was not surprising from our point of view ;) ). }

First of all, we give the first example of a Cayley graph with vanishing uniqueness thresholds for Bernoulli--Voronoi percolation.

\begin{theorem} \label{thm:mainTrees}
    Let $k\ge 2$, $d\ge 3$ and let $G\coloneqq\mathbb T_d\times\ldots\times \mathbb T_d$ be the direct product of $k$ copies of the $d$-regular tree $\mathbb{T}_d$.
    Then $p_u(\lambda)\to 0$ as $\lambda\to 0$.
\end{theorem}

We also address the question about vanishing uniqueness thresholds for the following concrete examples of higher rank symmetric spaces (see Theorem~\ref{thm:mainSymm} for a general result). 

\begin{theorem} \label{thm:mainHyp} Let $k\ge2$, $d_1,\ldots,d_k\ge2$ and let $M\coloneqq \mathbb H_{d_1}\times \ldots \times \mathbb H_{d_k}$ be the Riemannian product of hyperbolic spaces $\mathbb H_{d_i}$ of dimension $d_i$, for $i=1,\ldots, k$. Then $p_u(\lambda)\to0$ as $\lambda\to0$.
\end{theorem}

In fact, combining results in this paper with the main result in~\cite{GR25}, we are also able to address the general connected and simply connected case.

\begin{theorem} \label{thm:mainSimplyConn}
Let $M$ be the symmetric space of a connected and simply connected higher rank semisimple real Lie group $G$. Then $p_u(\lambda)\to0$ as $\lambda\to0$.
\end{theorem}

Notice that Theorem~\ref{thm:mainHyp} is not a special case of Theorem~\ref{thm:mainSimplyConn} which requires the Lie group itself to be simply connected. Examples of simply connected Lie groups are provided by universal covering groups.

We emphasize that each of the above theorems can be used to produce interesting new examples of non-amenable groups with the FIID sparse unique infinite cluster property, see Section~\ref{sec:FSUICP} for details.

\subsection{Vanishing uniqueness thresholds in products}

We now describe our method and its most general applications.

As observed in~\cite{GR25}, the main difficulty in proving smallness of $p_u$ is that the existence of a unique unbounded cluster is a non-local phenomenon. This is reflected in the reduction
\begin{equation} \label{equ:LROpu}
p_u(\lambda) = \inf \Big\{ p \, : \, \inf_{x,y} \mathbb P\Big(x \xleftrightarrow{\omega_p^{(\lambda)}} y\Big)>0 \Big\},
\end{equation}
proved for Poisson--Voronoi percolation in symmetric spaces in~\cite[Theorem~1.9]{GR25}. The argument extends to Bernoulli--Voronoi percolation on locally finite, connected, transitive, unimodular graphs (see Theorem~\ref{thm:UniquenessLRO}). Equ.~\eqref{equ:LROpu} characterizes the uniqueness phase by a positive lower bound on the two-point function uniformly over all pairs of points in the space, reflecting the non-local nature. If such a two-point function lower bound exists, we say that there is {\bf long-range order}. Our method proves long-range order at sufficiently small intensities by combining two main ingredients described below.

The main ingredient compared to the approach in \cite{GR25} is to find a suitable substitute for property~(T). We will work with the condition that the underlying space decomposes, roughly speaking, as a product of a {\em non-amenable} space with an {\em unbounded} space. We prove two technical results, Theorem~\ref{thm:LocalToGlobal} and Theorem~\ref{thm:LocalToGlobalSymm}, which allow us to improve sufficiently strong local information about Bernoulli--Voronoi percolation, resp.~Poisson--Voronoi percolation, to long-range order. At this point, the reader may find it instructive to consider the non-amenable Cayley graph $\mathbb T_3$ in which no non-trivial version of local uniqueness is sufficient for global uniqueness. 

%Question for the future: Is it true that if $X$ and $Y$ are symmetric spaces of non-compact connected semisimple real Lie groups $G$ and $H$, then G\times H$ is a Lie group asssociated to the symmetric space $X\times Y$? Since $G\times H$ has (T) iff both $G$ and $H$ have (T), is the same statement true for every Lie group associated to $X\times Y$?

We also emphasize that this non-amenable product situation is essentially orthogonal to the property~(T) case. Connected and simply connected semisimple real Lie groups are isomorphic to products of simple Lie groups. Those without property~(T) contain a factor isomorphic to $\SO(n,1)$ or $\SU(n,1)$, rank $1$ groups without property~(T) (see e.g.~the introduction to \cite{BDLHV}). Hence such a semisimple Lie group without property~(T), if it is higher rank, must be isomorphic to a product of at least two factors, at least one of which does not have property~(T). In this case, its symmetric space decomposes as the Riemannian product of the symmetric spaces of its factors, i.e.~has a non-amenable product structure.

The second main ingredient is composed of recent results about low-intensity limits of Poisson--Voronoi {\em tessellations}, also used in \cite{GR25}. Poisson--Voronoi tessellations are central to stochastic geometry, see e.g.~\cite{SW08}. Probabilistic properties of Poisson--Voronoi tessellations when the underlining metric space is a Riemannian symmetric space have been studied e.g.~in~\cite{CCE,BPP18,P18,BKP21,M24}. Recently, it was discovered that in the low-intensity limit (i.e.,~$\lambda\to0$) in certain geometries a non-trivial limiting object arises, called the {\bf ideal Poisson--Voronoi tessellation} (IPVT). For instance, in Euclidean space the limiting tessellation is trivial (in the sense that it is one cell which is the whole space) while in hyperbolic space there is a limiting tessellation consisting of countably many unbounded, one-ended cells. 

Existence and uniqueness when the underlying metric spaces are discrete trees were first established by Bhupatiraju \cite{B19}. Then the existence of a non-trivial IPVT of the hyperbolic plane (called there ``pointless Voronoi tessellation'') was used by Budzinski, Curien and Petri \cite{BCP22} to establish upper bounds on the Cheeger constant of hyperbolic surfaces in large genus. 
In~\cite{DCELU23}, the first author, Curien, Enriquez, Lyons and {\"U}nel proved a deterministic theorem giving sufficient conditions for the Voronoi diagram to converge (in the Fell topology) to an ideal Voronoi diagram when $\lambda \to 0$ (see~\cite[Theorem 2.3]{DCELU23}). They then considered two examples: real hyperbolic space $\mathbb H_d$, $d\ge2$, in which case the diagram is also a tessellation, and provided a systematic study of ${\rm IPVT(\mathbb{H}_d)}$ based on a surprisingly simple Poissonian description of the zero cell (i.e., the size-biased typical cell); and ${\rm IPVT(\mathbb{T}_d)}$, whose study was partially started in \cite{B19}. As a further application to illustrate~\cite[Theorem 2.3]{DCELU23} beyond Riemannian manifolds, 
the first author built and studied (certain features of) the IPVT of the $L^1$-product of two hyperbolic planes in~\cite{D24}. 
Finally, the IPVT was constructed directly and studied for symmetric spaces and certain products of trees in~\cite{FMW23} (without a proof that it is the unique low-intensity limit),
where it was used to prove that higher rank semisimple real Lie groups (and their lattices) have fixed price $1$, partially resolving a conjecture of Gaboriau~\cite{Gaboriau2000}.
The fixed price $1$ result depends on a remarkable property of the IPVT, proved in \cite{FMW23} in the higher rank case but also in~\cite{D24}, and described in the next paragraph.
%In that work, a remarkable property was discovered in higher rank (see the next paragraph), which is crucial in their proof that (lattices in) higher rank semisimple real Lie groups have fixed price $1$, resolving a conjecture of Gaboriau~\cite{Gaboriau2000}.  

The IPVT on higher rank symmetric spaces and on $L^2$-products of regular trees of degree at least three was shown in~\cite{FMW23} to have the property that every pair of cells shares an unbounded boundary, almost surely. We will refer to this property as the {\bf unbounded borders phenomenon}. In particular, the neighboring relation of the Delaunay graph on the IPVT is the countable complete graph. A finitary analogue of this statement was proved in~\cite{GR25}. More precisely, it follows from~\cite[Theorem~1.19 \& Proposition 4.16]{GR25} that for every $N\ge1$, there exists $R>0$ such that the neighboring relation of the Poisson--Voronoi tessellation restricted to the ball of radius $R$ is a complete graph of size at least $N$ with high probability uniformly in sufficiently small $\lambda$. We use this result in our applications to symmetric spaces.

Having sketched our method for the continuum case, here is the promised most general application to this setting.

\begin{theorem}[\textsc{Products of symmetric spaces}]\label{thm:mainSymm}
 Let $X$ and $Y$ be symmetric spaces of non-compact connected semisimple real Lie groups $G$ and $H$. Then Poisson--Voronoi percolation on the Riemannian product $M\coloneqq X\times Y$ satisfies $p_u(\lambda)\to0$ as $\lambda\to0$. 
\end{theorem}

Theorem~\ref{thm:mainHyp} follows directly from Theorem~\ref{thm:mainSymm} and further natural examples beyond property~(T) are easily obtained by taking products with hyperbolic spaces.

\subsection{Remarks on the discrete case}\label{sec:introdiscrete}

The discrete situation is more subtle. As mentioned above, the unbounded borders phenomenon holds in $L^2$-products of trees of degree at least three \cite{FMW23}. However, for Bernoulli--Voronoi percolation we are interested in the graph product. One of our results is an elementary proof of the unbounded borders phenomenon for $k$-fold graph products (i.e., $L^1$-products) of regular trees of the same degree; see Theorem~\ref{thm:UnboundedTouching}. In fact, it turns out that the unbounded borders phenomenon is not valid for the same class of examples considered in \cite{FMW23} when equipped with the graph distance, which can be seen already for $\mathbb T_3\times\mathbb T_4$. Roughly speaking, the difference between the metrics arises as follows. In the low intensity limit on graph products whose factors have different growth rates, points of Poisson point processes do not escape to the boundary in {\em all} coordinates. This is a discrete counterpart to~\cite[Theorem 2.1]{D24}, which states that the ``boundary part'' of the corona measure is supported on a proper subset of the Gromov boundary (horofunction boundary) of $\mathbb{H}_2\times\mathbb{H}_2$. But this property is crucial: Consider two sequences of $L^1$-horofunctions centered around points which converge to infinity in such a way that, say, the second coordinates stays bounded. These will not agree on increasingly large level sets. When all factors are given by the same tree, we show that this problem does not arise and that the unbounded borders phenomenon holds (see Theorem~\ref{thm:UnboundedTouching}). 

To prove Theorem~\ref{thm:mainTrees} given the results in Section~\ref{sec:LocalToGlobal}, it suffices to establish local uniqueness for fixed $p$ and small intensity $\lambda$, which we deduce from the unbounded borders phenomenon for the IPVT. Although the metric (i.e., the $L^1$-product metric) is different here than in \cite{FMW23}, we prove similar statements along the way. First, we construct an invariant measure on the space of distance-like functions (we define this space in Section \ref{sec:Corona}) using the same method as in \cite[Section 3]{FMW23}. To prove unbounded borders for the IPVT for products of trees with equal degrees, we then explicitly use the structure of the graph instead of the Lie structure of the automorphism group as in~\cite[Section 5]{FMW23}. We do this in Section \ref{sec:Trees}. Proposition~\ref{pr:BasicCorona} and Theorem~\ref{thm:SubseqLimit} correspond to~\cite[Proposition 5.1, Lemma 5.4, Lemma 5.5]{FMW23}, and Proposition~\ref{prop:UnboundedStab} and Lemma~\ref{lm:HoweMoore} correspond to \cite[Proposition~5.8]{FMW23}. The remaining sketch of the proof of Theorem~\ref{thm:UnboundedTouching} is essentially the same as Theorem~6.1 in~\cite{FMW23}.

\begin{remark} \label{rem:Generalization}
Theorem~\ref{thm:mainTrees} can be extended to other direct products including, for $m\ge1$, $\mathbb T_4\times\mathbb T_4\times(\mathbb{T}_3)^m$, $\mathbb T_5\times\mathbb T_5\times\mathbb T_4\times\mathbb T_3$ and $\mathbb T_3\times\mathbb T_3\times \mathbb{Z}^m$. More generally, the proof can be extended to direct products $G\coloneqq (\mathbb T_d)^k\times H$, where $d\ge3$, $k\ge2$ and $H$ is a locally finite, connected, transitive, unimodular graph  such that the volume of the ball $B_r$ of radius $r$ in $H$ satisfies $|B_r|\in O((d-1)^r)$. See Remark~\ref{rem:GeneralProducts} for an outline of the necessary modifications of the proof. 
\end{remark}

\subsection{Applications to FIID  sparse unique infinite clusters}\label{sec:FSUICP}

We now describe applications of our results to some problems originating in ergodic theory. We start by recalling a definition: A finitely generated group~$\Gamma$ has the {\bf FIID sparse unique infinite cluster property} if there exists a Cayley graph of~$\Gamma$, denoted ${\rm Cay}(\Gamma)$, such that
$$
\inf \bigg\{ \int_{\omega\in \mathcal U(G)} {\rm deg}_\omega(1_{\Gamma}) \, d \mu(\omega) \colon \mu\in F_{{\rm IID}}(\Gamma,\mathcal U({\rm Cay}(\Gamma))) \bigg\} = 0,
$$
where $\mathcal U({\rm Cay}(\Gamma))$ is the set of subgraphs of ${\rm Cay}(\Gamma)$ with a unique infinite cluster, $F_{{\rm IID}}(\Gamma,\mathcal U({\rm Cay}(\Gamma)))$ is the set of probability measures on $\mathcal U({\rm Cay}(\Gamma))$ which are ($\Gamma$-equivariant) factors of iid~processes (FIID) on $\Gamma$ and $1_\Gamma$ denotes the identity of $\Gamma$. More concretely, this means that for every $\varepsilon>0$ there is a FIID construction of a random subgraph which has a unique infinite cluster but is also sparse in the sense that its average density is at most $\varepsilon$.

In addition to the probabilistic interest in this property, its relevance stems from the fact that it implies fixed price~$1$; we refer to~\cite[Section 1.1]{GR25} for detailed background. For the purposes of this introduction, it is instructive to think of the canonical sparse FIID model, namely Bernoulli bond percolation with survival probability~$p$ tending to zero. It is not difficult to see that bounded degree graphs have $p_c>0$, i.e.~there are only finite clusters for~$p$ sufficiently small. To produce not only sparse infinite clusters but unique ones, more elaborate constructions are thus needed. For one-ended amenable groups, this is possible~\cite{T19}; see also \cite{TT13,BRR23,L13} for related results. In the opposite direction, the following question was posed implicitly in groundbreaking work of Hutchcroft and Pete \cite[Remark~4.4]{HP2020} and explicitly by Pete and Rokob~\cite{PR25}.

\begin{question}[{\cite[Question~1.5]{PR25}}] \label{qu:FSUIC} Give interesting examples of non-amenable Cayley graphs with FIID sparse unique infinite clusters.
\end{question} 

In~\cite[Theorem 1.5]{GR25}, an answer to Question~\ref{qu:FSUIC} was given using the phenomenon of vanishing uniqueness thresholds in Poisson--Voronoi percolation on the symmetric space of a connected higher rank semisimple real Lie group $G$ with property~(T) to construct FIID sparse unique infinite clusters in co-compact lattices of~$G$. It was also observed in~\cite[Remark 1.7]{GR25} that vanishing uniqueness thresholds for the discrete Bernoulli--Voronoi model would immediately imply the FIID sparse unique infinite cluster property. Indeed, let $\varepsilon>0$, then there exists $\lambda_0>0$ with $p_u(\lambda_0)<\varepsilon/d$, where $d$ is the degree in the graph. By keeping all edges between black vertices, one obtains an FIID bond percolation with expected degree at most $\varepsilon$ and a unique infinite cluster. With this background in mind, we now describe how the results proved in this paper fit into the above picture.

\medskip

{\bf\noindent Applications in the discrete setting.} Theorem~\ref{thm:mainTrees} provides new examples of non-amenable Cayley graphs with the FIID sparse unique infinite cluster property. The main point, however, is that we prove vanishing uniqueness thresholds for the graph itself and thus provide an affirmative answer to Question~\ref{qu:FSUIC} using one of the most natural FIID models beyond Bernoulli percolation (note that Bernoulli percolation never works, as pointed out above). 

Moreover, we emphasize that Theorem~\ref{thm:mainTrees} does not only imply fixed price~$1$ but strengthens the main result proved in~\cite{FMW23} to a finitary version while avoiding certain technicalities. More precisely, in that paper the unbounded borders phenomenon is leveraged to construct, for every $\varepsilon>0$, a random spanning graph $H$ of $G$ with average degree at most $2+\varepsilon$ as a {\em weak limit} of factors of iid processes. As a further technical ingredient, their construction requires proving that IPVT cells are hyperfinite. Using our approach, we are able to construct $H$ as a factor of iid without passing to a limit. We also construct $H$ as a genuine subgraph of~$G$. Finally, the tessellations we employ have finite cells and we are thus able to avoid technicalities associated to proving hyperfiniteness of infinite limiting cells.

\medskip

{\bf\noindent Applications in the continuum setting.}  For the symmetric space $M$ of a connected semisimple real Lie group $G$, the phenomenon of vanishing uniqueness thresholds has two main applications proved in \cite[Section~8 \& Section~9]{GR25}:
\begin{itemize}
\item It implies that the unit intensity Poisson point process $\Pi$ on $G$ (i.e., the Poisson point process with intensity measure given by Haar measure on $G$) equipped with iid ${\rm Unif}[0,1]$ marks satisfies: For every $\varepsilon>0$, there exists a $G$-equivariant factor graph $\mathcal H$ of $\Pi$ with a unique infinite cluster and $\mathbb E[{\rm deg}_{\mathcal H(\Pi_0)}(1_G)]\le\varepsilon$, where $\Pi_0\coloneqq \Pi\cup\{1_G\}$ denotes the Palm version; we refer the reader to~\cite[Chapter 3]{LP2018} for background on Poisson processes and to~\cite[Chapter 9]{LP2018} for an introduction to Palm theory. 

Again, the above property is a finitary version of the main result proved in \cite{FMW23} where for every $\varepsilon>0$ a weak limit of factor graphs of the Poisson point process with average degree at most $2+\varepsilon$ is constructed. As alluded to in the discrete case, we work with small but fixed intensity $\lambda$ as opposed to the low intensity limit, thereby producing a factor graph and avoiding technicalities associated to proving hyperfiniteness. Remark that in the setting of Theorem~\ref{thm:mainSimplyConn} and Theorem~\ref{thm:mainSymm}, this property of the Poisson point process is new when the associated Lie group does not have property~(T). We also point out that an analogous statement holds for the Poisson point process on $M$.

\item It implies that Cayley graphs of co-compact lattices $\Gamma\subset G$ have the FIID sparse unique infinite cluster property. Since $G$ is well known to admit co-compact lattices and these lattices have property~(T) if and only if $G$ does, Theorem~\ref{thm:mainSymm} provides several new examples with the FIID sparse unique infinite cluster property.
\end{itemize}

\medskip

{\bf\noindent Acknowledgements}. M.D'A.~acknowledges support by the ERC Consolidator Grant SuperGRandMA (Grant No.~101087572) and by the ANR project LOUCCOUM (ANR-24-CE40-7809).
J.G. was supported by MSCA Postdoctoral Fellowships 2022 HORIZON-MSCA-2022-PF01-01 project BORCA grant agreement number 101105722, and by the Alexander von Humboldt Foundation in the framework of the Alexander von Humboldt Professorship of Daniel Kráľ endowed by the Federal Ministry of Education and Research.
For the purpose of open access, the authors have applied a CC BY public copyright licence to any author accepted manuscript arising from this submission.

\medskip
{\noindent\bf Notation.} We record here some notations used throughout.

Let $G=(V,E)$ be a locally finite, connected (undirected) graph. We  define the notation $V(G)\coloneqq V$ and $E(G)\coloneqq E$ for the vertex and edge sets. We use the notation $e=[u,v]$ for an edge $e$ with endpoints $u$ and $v$. In this case, we also write $u\sim v$ and say that $u$ and $v$ are neighbors. We endow $G$ with the {\bf graph distance} ${\rm dist}(\cdot,\cdot)={\rm dist}_G(\cdot,\cdot)$ which assigns to vertices $u,v\in V$ the length of a shortest path between them. We denote by $B_R(v)$ the ball of radius $R\ge1$ around $v\in V$, that is all vertices $u\in V$ such that ${\rm dist}(u,v)\le R$. 

If $H$ is another locally finite, connected graph, the {\bf direct product} $G\times H$ of $G$ and $H$ is the graph  with vertex set $V(G)\times V(H)$ and edges between all vertices $(u_1,u_2)$ and $(v_1,v_2)$ such that $u_1=v_1$ and $u_2\sim v_2$, or $u_2=v_2$ and $u_1\sim v_1$. Note that 
$$
{\rm dist}_{G\times H}((u_1,v_1),(u_2,v_2))= {\rm dist}_G(u_1,u_2)+{\rm dist}_H(v_1,v_2),
$$
i.e.~the direct product is the $L^1$-product of the metric spaces $(G,{\rm dist}_G)$ and $(H,{\rm dist}_H)$.

A {\bf site percolation} on $G$ is a random subset $\omega$ of vertices and a {\bf bond percolation} is a random subset of edges. We denote by $\{u\overset{\omega}{\leftrightarrow} v\}$ the event that there exists a path between vertices $u$ and $v$ in the subgraph induced by $\omega$. If $\mu$ is the distribution of $\omega$, we simply write $\mu(u\leftrightarrow v)$.

Let $f$ and $g$ be positive real-valued functions defined on the same unbounded subset of $(0,\infty)$ (e.g.~$\mathbb \{1,2,\ldots\}$). We write $f(x)\in O(g(x))$ as $x\to\infty$ if there exists $C>0$ such that $f(x)\le Cg(x)$ for all sufficiently large $x$. We write $f(x)\in o(g(x))$ as $x\to\infty$ if for every $c>0$, $f(x)\le cg(x)$ for all sufficiently large $x$. Finally, we write $f(x)\in\Theta(g(x))$ if $f(x)\in O(g(x))$ and $g(x)\in O(f(x))$.

\section{Bernoulli--Voronoi percolation on graphs}\label{sec:Basic}

In this section, we provide the formal definition of Bernoulli--Voronoi percolation and recall fundamental properties. 

\subsection{Description of the model} \label{sec:model}

Let $G=(V,E)$ be a locally finite, connected graph with distinguished root $o\in V$. Given $\lambda\in(0,1]$, let 
\begin{equation*}
\mathbf X^{(\lambda)}=\big\{(X_1^{(\lambda)},Y_1^{(\lambda)}),(X_2^{(\lambda)},Y_2^{(\lambda)}),\ldots\big\}
\end{equation*}
be such that 
\begin{itemize}
\item the sequence $X^{(\lambda)}\coloneqq \big\{X_1^{(\lambda)},X_2^{(\lambda)},\ldots\big\}$ is $\lambda$-Bernoulli site percolation ordered according to increasing distance from the root,
\item the sequence $Y^{(\lambda)}\coloneqq \{Y_1^{(\lambda)},Y_2^{(\lambda)},\ldots\}$ consists of iid ${\rm Unif}[0,1]$-labels and is independent of $X^{(\lambda)}$.
\end{itemize}
The associated {\bf Voronoi diagram} is defined to be 
\begin{equation} \label{def:VoronoiDiagram}
\mathrm{Vor}(\mathbf{X}^{(\lambda)})=\big\{C_1^{(\lambda)},C_2^{(\lambda)},\ldots\big\},
\end{equation}
where $C_i^{(\lambda)}\subset V$ consists of all vertices closer to $X_i^{(\lambda)}$ (w.r.t.~the graph distance) than to any of the other point of $\mathbf{X}^{(\lambda)}$ and, in case of a tie, vertices are allocated to the random points with minimal label $Y_i^{(\lambda)}$.

The choice of tie breaking rule implies that if $v\in V$ is allocated to $X_i^{(\lambda)}$, then every point on a shortest path between $v$ and $X_i^{(\lambda)}$ is allocated to $X_i^{(\lambda)}$ as well. In particular, $\mathrm{Vor}(\mathbf{X}^{(\lambda)})$ is a partition of $G$ into connected subsets $C_i^{(\lambda)}\subset V$ a.s. 

We will refer to each $C_i^{(\lambda)}$ as a {\bf cell.} Typically, cells are finite; see Proposition~\ref{prop:VoronoiBasics}. Given $p\in[0,1]$, let 
\begin{equation*}
\mathrm{Vor}(\mathbf{X}^{(\lambda)})_p = \big\{ B_1^{(\lambda)},B_2^{(\lambda)},\ldots \big\}
\end{equation*}
be obtained from $\mathrm{Vor}(\mathbf{X}^{(\lambda)})$ by independently retaining or deleting each cell with retention probability $p$, which we interpret as an independent black-and-white coloring 
where black cells correspond to retained cells 
%followed by retaining only the black cells 
$B_i^{(\lambda)}$. 
Let 
\begin{equation*}
\omega^{(\lambda)}_{p} = \bigcup_{i=1,2,\ldots} B_i^{(\lambda)}
\end{equation*}
denote the random subset of vertices consisting of all the vertices belonging to black cells. We refer to this site percolation model as $(\lambda,p)$-{\bf Bernoulli--Voronoi percolation} or simply as {\bf Bernoulli--Voronoi percolation} on $G$.

We denote by $\mathbb P_{p}^{(\lambda)}$ the distribution of $\omega^{(\lambda)}_{p}$. We write $\mathcal B_p = \mathrm{Vor}(\mathbf{X}^{(\lambda)})_p$ for the collection of black cells and
\vspace{1mm}
 \begin{equation*}
 \mathcal W_p \coloneqq  \mathrm{Vor}(\mathbf{X}^{(\lambda)}) \setminus \mathrm{Vor}(\mathbf{X}^{(\lambda)})_p = \big\{ W_1^{(\lambda)},W_2^{(\lambda)},\ldots \big\}
 \end{equation*}
 for the collection of white cells $W_i^{(\lambda)}$.
 We denote by 
 $C^{(\lambda)}(v)$ the Voronoi cell of $v\in V$ in $\mathrm{Vor}(\mathbf{X}^{(\lambda)})$, that is, the unique cell to which the vertex $v$ belongs.
A \textbf{cluster} is any connected component in $\omega^{(\lambda)}_{p}$.
 We write $\mathcal C_p^{(\lambda)}(v)$ for the cluster of $v$ in $\omega^{(\lambda)}_{p}$, which might be empty if $C^{(\lambda)}(v)\not\in \mathcal B_p$.

The following result summarizes basic properties, see e.g. \cite[Section~8.5]{LP16}.

\begin{proposition} \label{prop:VoronoiBasics}
    Let $G=(V,E,o)$ be a locally finite, connected, transitive, unimodular, rooted graph and let $\lambda\in(0,1]$, $p\in[0,1]$. Then
    \begin{itemize}
        \item[{\rm(i)}] all cells $C_i^{(\lambda)}$ are finite almost surely,
        \item[{\rm(ii)}] $\omega_p^{(\lambda)}$ satisfies the FKG-Inequality,
        \item[{\rm(iii)}] $\omega_p^{(\lambda)}$ defines an FIID percolation, the factor map being equivariant with respect to all automorphisms of $G$,
        \item[{\rm(iv)}] $\omega_p^{(\lambda)}$ is ergodic with respect to every subgroup $\Gamma\subset\mathrm{Aut}(G)$ with an infinite orbit.
    \end{itemize}     
\end{proposition}
%\MD{Is $\omega_p^{(\lambda)}$ also exponentially ergodic?} I guess we can keep this question for the future ;)
%\textcolor{orange}{provide reference for FKG inequality, define FIID}

\subsection{The uniqueness threshold} 

The {\bf critical probability} is defined to be 
\begin{equation*}
p_c(\lambda)=\inf \big\{ p \, : \, \omega^{(\lambda)}_{p} \, \, \text{has an infinite cluster w.p.p.} \big\}
\end{equation*}
and the {\bf uniqueness threshold} is defined to be
\begin{equation*}
p_u(\lambda)=\inf \big\{ p \, : \, \omega^{(\lambda)}_{p} \, \, \text{has a unique infinite cluster w.p.p.} \big\}.
\end{equation*}

The following characterization of the uniqueness phase in terms of long-range order goes back essentially to \cite{LS99}.

\begin{theorem}[\textsc{Uniqueness and long-range order}]  \label{thm:UniquenessLRO}
Let $G=(V,E,o)$ be a locally finite, connected, transitive, unimodular, rooted graph. Then 
\begin{equation} \label{equ-UniquenessLRO}
p_u(\lambda) = \inf \Big\{ p \, : \, \inf_{u,v\in V} \mathbb P_{p}^{(\lambda)}( u \leftrightarrow v)>0 \Big\}.
\end{equation}
\end{theorem}

Note that Voronoi percolation is not insertion tolerant (and also not deletion tolerant). Therefore Theorem \ref{thm:UniquenessLRO} does not follow directly from \cite{LS99}. However, as observed in \cite{GR25}, it is possible to use a form of insertion tolerance by viewing Voronoi percolation as Bernoulli percolation on the Delaunay graph. Details are collected in Appendix~\ref{app:Proofs}. The argument presented there also implies the following result.

\begin{proposition}[\textsc{Phase transition of Voronoi percolation}] \label{prop:Phases}
Let $G=(V,E,o)$ be a locally finite, connected, transitive, unimodular, rooted graph and let $\lambda\in(0,1]$, $p\in[0,1]$. Then the following hold:
\begin{itemize}
\item[{\rm(i)}] For $p<p_c(\lambda)$, $\omega^{(\lambda)}_p$ has no infinite cluster a.s.
\item[{\rm(ii)}] For $p_c(\lambda) < p < p_u(\lambda)$, $\omega^{(\lambda)}_p$ has infinitely many infinite clusters a.s.
\item[{\rm(iii)}] For $p>p_u(\lambda)$, $\omega^{(\lambda)}_p$ has a unique infinite cluster a.s.
\end{itemize}
\end{proposition}

\subsection{Description of the continuum model}

We briefly recall the definition of Poisson--Voronoi percolation for the convenience of the reader and refer to \cite{GR25} for details. 

In this paper, we will only consider Poisson--Voronoi percolation on non-discrete spaces and may thus use the same notation as in the discrete case without risk of confusion. 

Let $(M,d,o,\mu)$ be a proper geodesic metric space with some fixed origin $o\in M$ and infinite Radon measure $\mu$. Let $\lambda>0$ and $p\in(0,1]$. Let
$$
\mathbf X^{(\lambda)}=\big\{(X_1^{(\lambda)},Y_1^{(\lambda)}),(X_2^{(\lambda)},Y_2^{(\lambda)}),\ldots\big\}
$$
be such that 
\begin{itemize}
\item[--] the sequence $X^{(\lambda)}\coloneqq \big\{X_1^{(\lambda)},X_2^{(\lambda)},\ldots\big\}$ is a Poisson point process on $M$ with intensity $\lambda \cdot \mu$ ordered according to increasing distance from the origin.
\item[--] the sequence $Y^{(\lambda)}\coloneqq \{Y_1^{(\lambda)},Y_2^{(\lambda)},\ldots\}$ consists of iid ${\rm Unif}[0,1]$-labels and is independent of $Y^{(\lambda)}$.
\end{itemize}
The definition of the associated {\bf Voronoi diagram} and (Voronoi) {\bf cells} is the same as before. The same independent black-and-white coloring procedure followed by retaining the black cells leads to the definition of a random closed subset of $M$ consisting of all points which belong to black cells. We refer to this continuum percolation model as $(\lambda,p)$-{\bf Poisson-Voronoi percolation} or simply as {\bf Poisson-Voronoi percolation} on $M$ and we refer to the path connected components as {\bf clusters}. 

The natural analogues of Proposition~\ref{prop:VoronoiBasics}, Theorem~\ref{thm:UniquenessLRO} and Proposition~\ref{prop:Phases} hold as shown in~\cite{GR25}.

\section{From local to global uniqueness in products} \label{sec:LocalToGlobal}

This section is devoted to two technical results, Theorem~\ref{thm:LocalToGlobal} and Theorem~\ref{thm:LocalToGlobalSymm}. In Section~\ref{sec:GraphProducts}, we show Theorem~\ref{thm:LocalToGlobal} which roughly speaking asserts that for Bernoulli--Voronoi percolation on non-amenable products of infinite Cayley graphs, having information about {\em local uniqueness} at sufficiently small intensities is sufficient to establish {\em long-range order} at sufficiently small intensities. This result also extends to products of more general infinite graphs, see Section~\ref{sec:BeyondCayley}. In Section~\ref{sec:SymmProducts}, we show Theorem~\ref{thm:LocalToGlobalSymm}, which is a matching result in the continuum setting for Poisson--Voronoi percolation on Riemannian products of symmetric spaces.

\medskip

{\bf\noindent Convention:} To shorten notation, we make the convention that all Cayley graphs in this section are Cayley graphs of finitely generated groups with respect to some finite generating sets.

\subsection{Non-amenable graph products} \label{sec:GraphProducts}

Let $G=(V,E,o)$ be a locally finite, connected, transitive, unimodular, rooted graph. Note that we used the root only to enumerate the Bernoulli process in the definition of Bernoulli--Voronoi percolation. For transitive graphs and the purposes of this subsection, the particular choice of root will thus not play a role and we therefore omit it from the notation throughout.

\begin{figure}[!hbtp]
    \hspace{5pt}
\includegraphics[width=0.75\linewidth]{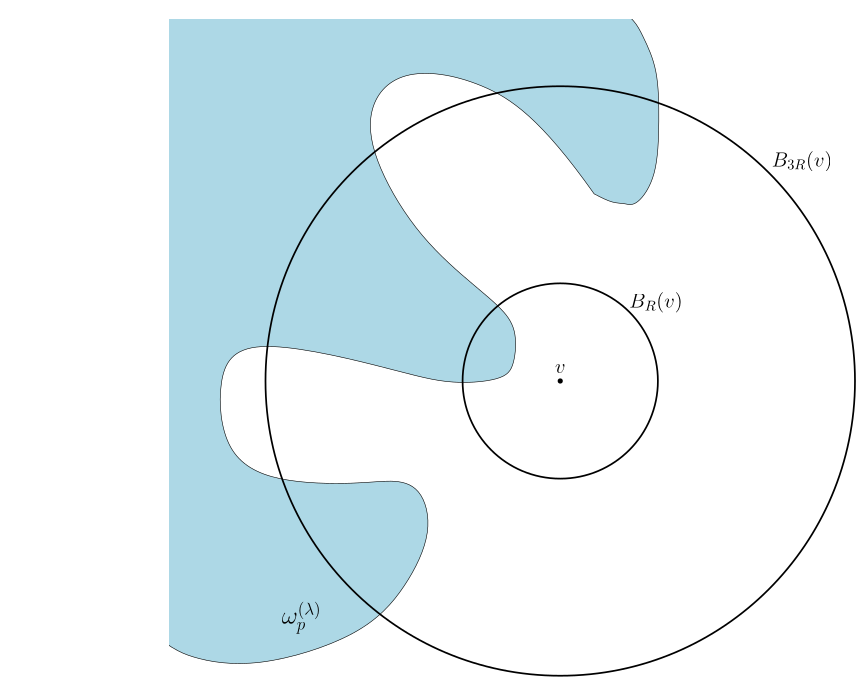}
    \caption{Portrait of a ``good'' event $A(\lambda,p,R,v)$.}
    \label{fig:alprv}
\end{figure}

We now formulate a suitable notion of ``good'' events for local uniqueness. For $\lambda\in(0,1]$, $p\in(0,1]$, $R\ge1$ and $v\in V$, let $A(\lambda,p,R,v)$ denote the event that $\omega_p^{(\lambda)}$ intersects $B_R(v)$ with at least one cluster and, moreover, $\omega_p^{(\lambda)}\cap B_{3R}(v)$ 
is contained in a single cluster of $\omega_p^{(\lambda)}$ (see Fig.~\ref{fig:alprv} for a portrait). The reason we require the cluster to be unique in a ball of radius $3R$ is that with this choice, if the ``good'' event $A(\lambda,p,R,\cdot)$ holds at $v$ and $v'\sim v$ simultaneously, then these events are witnessed by the same cluster.

We say that Bernoulli--Voronoi percolation on $G$ has {\bf local uniqueness at low intensities} if for every $p_0\in(0,1]$ and $\varepsilon>0$, there exist $R>0$ and $\lambda_0\in(0,1)$ such that
\begin{equation}\label{equ:localuniqueness}
\inf_{\lambda\le\lambda_0} \mathbb P\big( A(\lambda,p_0,R,o) \big)>1-\varepsilon.
\end{equation}
We are now in a position to state the main result of this subsection. Recall that $\{u\leftrightarrow v\}$ denotes the event that $u$ is connected to $v$ in the percolation subgraph. 

\begin{theorem} \label{thm:LocalToGlobal}
Let $X$ and $Y$ be infinite Cayley graphs. Suppose that $X$ is non-amenable and that Bernoulli--Voronoi percolation on $G=X\times Y$ has local uniqueness at low intensities. Then Bernoulli--Voronoi percolation on $G$ satisfies the following: for every $p_0\in(0,1]$, there exists $\lambda_0>0$ such that  
$$
\inf_{u,v\in V(X\times Y)} \mathbb P_{p_0}^{(\lambda)}( u\leftrightarrow v) >0
$$
for all $\lambda<\lambda_0$.
\end{theorem}

The proof of this result is inspired by the proof of non-uniqueness at $p_u$ for Bernoulli percolation on non-amenable products in \cite{P00}. It will be important for the argument that both factors are infinite and that one of them is non-amenable. Non-amenability will be used in form of the following version of the celebrated quantitative characterization of non-amenability due to~\cite{BLPS99}.

\begin{theorem}[\textsc{Threshold for infinite clusters}] \label{theorem:threshold} Let $X$ be a Cayley graph of a non-amenable group $\Gamma$. Then there exists $p^*\in[0,1)$ such that every $\Gamma$-invariant site percolation $\omega$ on $X$ with $\mathbb P(o\in \omega)>p^*$ satisfies 
$$
\mathbb P( |C_\omega(o)|=\infty)\ge 2/3,
$$ 
where $C_\omega(o)$ denotes the $\omega$-cluster of the origin $o\in V(X)$.
\end{theorem}

\begin{proof} This follows from the proof of \cite[Theorem 2.12]{BLPS99}. We include a sketch of the argument for the convenience of the reader.

For $A\subset V(X)$, let $\partial_V A\coloneqq  \{v\in V(X)\setminus A: v\sim u \text{ for some } u\in A \}$ and $\partial_E A\coloneqq  \{e\in E(X): e=[u,v] \text{ for some } u\in A, v\in V(X)\setminus A \}$ denote the vertex boundary and the edge boundary of $A$. Recall that the edge-isoperimetric constant of $X$ is defined to be
$$
\iota_E(X)\coloneqq  \inf \bigg\{ \frac{|\partial_E K|}{|K|} : \emptyset \neq K \subset V(X) \text{ finite} \bigg\}, 
$$
which is strictly positive if and only if $\Gamma$ is non-amenable, see for example \cite{BLPS99}. To shorten notation, set $d\coloneqq {\rm deg}_X(o)$.

Let $\omega$ be the configuration of a $\Gamma$-invariant site percolation on $X$. Without loss of generality, assume $\mathbb P(C_\omega(o)=V(X))<1$. For $u,v\in V(X)$, we denote by $\eta(u,v)\coloneqq |\{x\in C_\omega(u): x\sim v\}|$ be the number of neighbors of $v$ in the $\omega$-cluster of~$u$. Note that  
$$
\sum_{v\in\partial_V C_\omega(u)} \eta(u,v) = |\partial_E C_\omega(u)|.
$$
Now consider the $\Gamma$-diagonally-invariant mass transport function
$$
f(u,v) \coloneqq  \mathbb E \bigg[ \frac{\mathbf 1\{|C_\omega(u)|<\infty,v\in\partial_V C_\omega(u)\} \eta(u,v)}{|\partial_E C_\omega(u)|} \bigg].
$$
Since the total amount of mass leaving $o$ is $\mathbb P[o\in \omega, |C_\omega(o)|<\infty]$ and the total amount of mass entering $o$ is bounded by $d(1-\mathbb P(o\in \omega))/\iota_E(X)$ (in fact, this is a strict inequality \cite{BLPS99}, but we shall not need this fact), the Mass Transport Principle implies that
$$
\mathbb P(|C_\omega(o)|=\infty) \ge \mathbb P(o\in\omega) \bigg(1+\frac{d}{\iota_E(X)}\bigg)-\frac{d}{\iota_E(X)}.
$$
Note that the right-hand side tends to one as $\mathbb P(o\in\omega)\to1$. The claim follows.
\end{proof}

\begin{proof}[Proof of Theorem~\ref{thm:LocalToGlobal}] Let $\omega_p^{(\lambda)}$ denote Bernoulli--Voronoi percolation with intensity $\lambda\in[0,1]$ and survival probability $p\in[0,1]$ on $X\times Y$. 

Fix $p_0\in(0,1)$. We will show that there exists $\lambda_0>0$ such that for all $\lambda<\lambda_0$ and every $\delta>0$
\begin{equation}\label{equ:LROsprinkled}
\inf \Big\{ \mathbb P_{p_0+\delta}^{(\lambda)} \big( (x_1,y_1) \leftrightarrow (x_2,y_2) \big) \, : \, x_1,x_2\in V(X), y_1,y_2\in V(Y) \Big\}>0,
\end{equation}
which implies the claim.

\medskip

{\em Step 1} (Pairs of vertices on different $Y$-levels). To prove \eqref{equ:LROsprinkled}, we first prove a uniform lower bound for points of the form $(o_X,y_1)$ and $(o_X,y_2)$, where $o_X$ is a fixed origin in $V(X)$, 
using a version of the ``shadowing method'' of \cite{PP00,P00}. This step will use non-amenability of $X$.

Let $p^*$ be as in Theorem \ref{theorem:threshold}. Since $p^*<1$, the triangle inequality and \eqref{equ:localuniqueness} imply that there exist $R>0$ and $\lambda_0>0$ such that for all $x\in V(X)$ and all $y,y'\in V(Y)$,
\begin{equation} \label{equ:aux perc marginal}
\inf_{\lambda\le\lambda_0} \mathbb P\big( A(\lambda,p_0,R,(x,y)) \cap A(\lambda,p_0,R,(x,y')) \big)>p^*.
\end{equation}
In particular, this holds with $y=y_1$ and $y'=y_2$ and the above lower bound does not depend on this choice. 

Fix $R$ and $\lambda_0$ as above and let $\lambda<\lambda_0$.
We define an auxiliary site percolation model $\xi$ on $X$ as follows: $x\in V(X)$ belongs to $\xi$ if and only if both $A(\lambda,p_0,R,(x,y_1))$ and $A(\lambda,p_0,R,(x,y_2))$ occur, i.e.~we require the local uniqueness event to hold for the copies of $x$ on both levels $X\times\{y_1\}$ and $X\times\{y_2\}$ simultaneously. By \eqref{equ:aux perc marginal} and Theorem \ref{theorem:threshold},
$$
\mathbb P( |C_\xi(o_X)|=\infty) \ge 2/3.
$$
Thus with probability at least $2/3$, there exists an infinite path $\rho\coloneqq  (\rho_1,\rho_2,\rho_3,\ldots)$ in $V(X)$ starting with $\rho_1=o_X$ such that the events $A(\lambda,p_0,R,z)$ occur for all points 
%$z$ of the form $\rho_i\times \{y_1\}$ and $\rho_i\times\{y_2\}$. 
$z\in\{(\rho_i,y_1),(\rho_i,y_2) : i=1,2,\ldots\}$. 

In particular, the paths $\rho\times\{y_1\}$ and $\rho\times\{y_2\}$ in $X\times Y$ are such that within distance $R$ from each of their members is an $\omega_{p_0}^{(\lambda)}$-cluster, which moreover is unique within a ball of radius $3R$ centered at each path step.
But this implies that there are two $\omega_{p_0}^{(\lambda)}$-clusters, one in the $R$-neighborhood of $\rho\times\{y_1\}$ and one in the $R$-neighborhood of $\rho\times\{y_2\}$, such that there are infinitely many disjoint paths of length at most $d_Y(y_1,y_2)+2R$ between them in the Delaunay graph associated to the Bernoulli--Voronoi tessellation (here we used that every Voronoi cell is finite almost surely by Proposition \ref{prop:VoronoiBasics}~(i) to guarantee disjointness,
and that the distance between two cells in the Delaunay graph is bounded by their graph distance), 
and, moreover, these two clusters are the unique clusters seen from members of the respective path in the $3R$-neighborhood. This implies that for every $\lambda<\lambda_0$ and every $\delta>0$, we have that
$$
\mathbb P_{p_0+\delta}^{(\lambda)} \big( B_R((o_X,y_1)) \leftrightarrow B_R((o_X,y_2)) \big) \ge 2/3.
$$
By invariance and the fact that the above bound did not depend on $y_1,y_2$, we obtain that for every $\lambda<\lambda_0$ and every $\delta>0$
\begin{equation} \label{equ:TPF balls}
    \inf \Big\{ \mathbb P_{p_0+\delta}^{(\lambda)} \big( B_R((x,y)) \leftrightarrow B_R((x,y')) \big) \Big\} \ge 2/3,
\end{equation}
where the infimum runs over all $x\in V(X), y,y'\in V(Y)$.

\medskip

{\em Step 2} (General pairs of vertices). Now consider an arbitrary pair of vertices $(x_1,y_1)$ and $(x_2,y_2)$. This step will use the fact that $Y$ is infinite.

For every $y\in V(Y)$, consider the event 
$$
H(y,\lambda,\delta) \coloneqq  \Big\{ B_R((x_1,y_1)) \leftrightarrow B_R((x_1,y)), B_R((x_2,y_2)) \leftrightarrow B_R((x_2,y)) \text{ in } \omega_{p_0+\delta}^{(\lambda)} \Big\}.
$$ 
By \eqref{equ:TPF balls}, we have that $\mathbb P(H(y,\lambda,\delta))\ge1/3>0$ for every $y\in V(Y)$, $\lambda<\lambda_0$, $\delta>0$. Using that $\mathbb P(\limsup A_n) \ge \limsup \mathbb P(A_n)$ for any sequence of events $A_n$, we obtain that for every $\lambda<\lambda_0$ and $\delta>0$
$$
\mathbb P \big( H(y,\lambda,\delta) \text{ for infinitely many } y \big) \ge 1/3.
$$
Since there are only finitely many clusters of $\omega_{p_0+\delta}^{(\lambda)}$ intersecting a fixed ball, we conclude that on the event 
$$
H\coloneqq  \{ H(y,\lambda,\delta) \text{ for infinitely many } y \},
$$
there exist two $\omega_{p_0+\delta}^{(\lambda)}$-clusters, one intersecting $B_R((x_1,y_1))$ and the other one intersecting $B_R((x_2,y_2))$, which come infinitely often within bounded distance of one another. Thus another sprinkling (i.e., increasing the percolation parameter by another small $\delta>0$)
will connect them almost surely and, noting that the pair of points was arbitrary, we obtain that for every $\lambda<\lambda_0$ and every $\delta>0$
$$
\inf \Big\{ \mathbb P_{p_0+\delta}^{(\lambda)} \big( B_R((x,y)) \leftrightarrow B_R((x',y')) \big) \, : \, x,x'\in V(X), y,y'\in V(Y) \Big\} \ge 1/3.
$$
The FKG-inequality (Proposition~\ref{prop:VoronoiBasics}~(ii)) together with automorphism invariance now implies that for every $\lambda<\lambda_0$ and $\delta>0$
$$
\inf \Big\{ \mathbb P_{p_0+\delta}^{(\lambda)} \big( (x,y) \leftrightarrow (x',y') \big) \, : \, x,x'\in V(X), y,y'\in V(Y) \Big\} > 0.
$$
This shows \eqref{equ:LROsprinkled} and the proof of Theorem~\ref{thm:LocalToGlobal} is complete.
\end{proof}

\subsubsection{Beyond Cayley graphs} \label{sec:BeyondCayley} Theorem~\ref{thm:LocalToGlobal} may be extended to more general classes of graphs. We record these cases below for completeness, but have included the detailed proof above only for the most concrete case for our applications.

\medskip

{\bf (a).}
It extends to the case that $X$ and $Y$ are infinite, locally finite, connected, transitive, unimodular graphs, $X$ is non-amenable and Bernoulli--Voronoi percolation on $X\times Y$ has local uniqueness at low intensities. The proof goes through virtually unchanged given Proposition~\ref{prop:VoronoiBasics} and the straightforward analogue of Theorem~\ref{theorem:threshold} for ${\rm Aut}(X)$-invariant site percolations.

\medskip

{\bf (b).}
In the setting of (a) above, the assumption of transitivity may be relaxed to quasi-transitivity as long as we require local uniqueness at low intensities for a complete set of orbit representatives; using the Mass Transport Principle~\cite[Corollary 3.5]{BLPS99}, the modifications are straightforward.

\medskip

{\bf (c).}
In the setting of (a), we may remove the assumption of unimodularity as long as $X$ has a closed, transitive subgroup $\Gamma$ of automorphisms such that $\Gamma$ is a non-amenable group; this is because Bernoulli--Voronoi cells are still a.s.~finite in this case (see Lemma~\ref{lm:Finite2} below) and Theorem~\ref{theorem:threshold} holds in this case for $\Gamma$-invariant site percolations by~\cite[Theorem 2.3]{P00} with a similar proof using the tilted Mass Transport Principle~\cite[Lemma 2.4]{P00}. We emphasize that this argument requires the group to be non-amenable and not the graph. 

\begin{lemma} \label{lm:Finite2}
    Let $G=(V,E)$ be a locally finite, connected, transitive graph, and $\lambda>0$. Then all cells $C_i^{(\lambda)}$ are finite almost surely.
\end{lemma}
\begin{proof}
It suffices to show that $C^{(\lambda)}(o)$ is finite a.s. 

Let $\varepsilon>0$. Choose $s>0$ such that $\mathbb P(B_s(o)\cap X^{(\lambda)} \ne \varnothing)>1-\varepsilon$. Set $f(r):=|B_r(o)|$ for $r\ge1$. If $v\in V$ with $d(o,v)=r\ge 2s$, then $v\notin C^{(\lambda)}(o)$ if there exists $i\in\mathbb N$ and $z\in V$ such that $d(o,z)=r$ and $\{v, X_i^{(\lambda)}\}\subset B_{r/4}(z)$. For each $z\in V$, 
$$
\mathbb P( X^{(\lambda)} \cap B_{r/4}(z) = \emptyset) = (1-\lambda)^{f(r/4)}.
$$ 
Moreover, a standard packing argument implies that there is a set $Z\subset V$ of size at most $f(5r/4)/f(r/8)$ such that 
$$
Z\subseteq \{z \in V:d(z,v)=r\}\subseteq \bigcup _{z\in Z} B_{r/8}(z).
$$
By a union bound, we have that
$$
\mathbb{P}\Big(C^{(\lambda)}(o) \not \subseteq B_r(o) \Big) \le \varepsilon + \frac{f(5r/4)}{f(r/8)} (1-\lambda)^{f(r/4)}.
$$
We claim that we may bound the right-hand side by~$2\varepsilon$ by choosing a suitable $r$ as follows. Either there is a sequence $r_i\to \infty$ such that $f(r_i/4)\ge r_i^2$ and then
$$\frac{f(5r_i/4)}{f(r_i/8)} (1-\lambda)^{f(r_i/4)}\le d^{5r_i/4}(1-\lambda)^{r_i^2}\xrightarrow{i\to\infty} 0,$$
where $d$ denotes the maximal degree in $G$.
Otherwise, for all sufficiently large $r$,
$$\frac{f(5r/4)}{f(r/8)} (1-\lambda)^{f(r/4)}\le (5r/4)^2(1-\lambda)^{r}\xrightarrow{r\to\infty} 0.$$
Since $\varepsilon>0$ was arbitrary, the claim follows.
\end{proof}

We may also point out that the threshold mentioned above fails for locally finite, connected, transitive graphs that are either amenable or non-unimodular (thus non-amenable) with an amenable automorphism group; we refer the interested reader to \cite{BLPS99,P00,H20,TT25}.

\subsection{Products of symmetric spaces} \label{sec:SymmProducts}

Let $X$ be the symmetric space of a non-compact connected semisimple real Lie group $G$. Let $\omega_p^{(\lambda)}$  denote Poisson--Voronoi percolation with intensity $\lambda>0$ and survival probability $p\in(0,1]$ on $X$. We now introduce a slightly modified version of local uniqueness. Namely, for $\lambda>0$, $p\in(0,1]$, $R>0$, $S>0$ and $x\in X$, let $A(\lambda,p,R,S,x)$ denote the event that $\omega_p^{(\lambda)}$ intersects $B_R(x)$ with at least one cluster and, moreover, $\omega_p^{(\lambda)}\cap B_{R+S}(x)$ is contained in a single cluster of $\omega_p^{(\lambda)}$. 

We say that Poisson--Voronoi percolation on $X$ has {\bf local uniqueness at low intensities} if for every $p_0\in(0,1]$ and $\varepsilon>0$, there exists $R>0$ with the following property. For every $S>0$, there exists $\lambda_0>0$ such that
\begin{equation}\label{equ:localuniquenessSymmSpace}
\inf_{\lambda\le\lambda_0} \mathbb P\big( A(\lambda,p_0,R,S,o) \big)>1-\varepsilon.
\end{equation}
The choice of $S=2R$ in the discrete case was sufficient for our argument, but not necessary. In the cases where we establish local uniqueness at low intensities (cf.~Theorem~\ref{thm:localuniqueness}), our arguments establish the same condition with any $0<S<\infty$, possibly depending on $R$. In the continuum setting we use this additional freedom to prove the following local-to-global uniqueness criterion.

\begin{theorem} \label{thm:LocalToGlobalSymm}
Let $X$ and $Y$ be symmetric spaces of non-compact connected semisimple real Lie groups $G$ and $H$. Suppose that Poisson--Voronoi percolation on the Riemannian product $M\coloneqq X\times Y$ has local uniqueness at low intensities. Then Poisson--Voronoi percolation on $M$ satisfies the following property. For every $p_0\in(0,1]$, there exists $\lambda_0>0$ such that 
$$
\inf_{x,y\in M} \mathbb P_{p_0}^{(\lambda)}(x\leftrightarrow y) >0
$$
for all $\lambda<\lambda_0$.
\end{theorem}

The proof of Theorem~\ref{thm:LocalToGlobalSymm} follows the main ideas of the proof of Theorem~\ref{thm:LocalToGlobal} with modifications for the non-discrete setting. We start with a preparation.

\begin{lemma}\label{lm:lattice}
Let $X$ be the symmetric space of a non-compact connected semisimple real Lie group $G$. Fix $o\in X$. Then there exists a Cayley graph ${\rm Cay}(\Gamma)=(\Gamma,E)$ of a finitely generated subgroup $\Gamma\subset G$ and constants $C,D\in(0,\infty)$ such that the following hold:
\begin{itemize}
\item[{\rm(i)}] $\Gamma$ is non-amenable,
\item[{\rm(ii)}] $\sup_{x\in X} \inf_{g\in \Gamma} d_X(go,x) \le C$,
\item[{\rm(iii)}] $\sup \big\{ d_X(go,ho) : g,h\in\Gamma, [g,h]\in E \big\} \le D$.
\end{itemize}
\end{lemma}
\begin{proof}
It is well-known that there exists a co-compact lattice $\Gamma\subset G$, see~\cite{BHC}. By the \v{S}varc--Milnor lemma, $\Gamma$ is finitely generated. Let ${\rm Cay}(\Gamma)=(\Gamma,E)$ be a Cayley graph of~$\Gamma$; we now show that it has the desired properties.

(i).~It is well-known that $G$ is non-amenable (see e.g.~\cite[Theorem 2.2~(i)]{GR25}) and hence every lattice in $G$ is non-amenable, see e.g.~\cite[Corollary G.3.8]{BDLHV}.

(ii).~The set $\{go: g\in\Gamma\}\subset X$ induces a Voronoi diagram $\{V_g\}_{g\in \Gamma}$ of $X$, where $V_g$ consists of all $x\in X$ for which $d_X(x,\gamma o)$ is minimized among $\gamma\in G$ by~$g$. Since $\Gamma$ is co-compact in $G$, \cite[Lemma 9.3~(2)]{GR25} shows that there exists $C>0$ such that $V_g\subset B_C(go)$ for every $g\in\Gamma$.

(iii).~Denote the finite generating set associated to ${\rm Cay}(\Gamma)$ by $S=\{s_1,\ldots,s_n\}$. We claim that 
$$
D\coloneqq \sup \{d_X(o,s_io) : i=1,\ldots,n\}
$$ 
works as required. Clearly $D<\infty$ and $D>0$ follows from the fact that $\Gamma$ is non-compact and \cite[Lemma 9.3~(1)]{GR25}. Finally, if $g,h\in \Gamma$ with $[g,h]\in E$, then $h=gs_i$ or $g=hs_i$ for some $i\in\{1,\ldots,n\}$. Since $G$ acts on $X$ by isometries (see e.g.~\cite[Theorem 2.2~(i)]{GR25}), in either case $d_X(go,ho)=d_X(o,s_io)\le D$.
\end{proof}

\begin{proof}[Proof of Theorem~\ref{thm:LocalToGlobalSymm}] Let $\omega_p^{(\lambda)}$ denote Poisson--Voronoi percolation with intensity $\lambda>0$ and survival probability $p\in(0,1]$ on $M=X\times Y$. Let $\Gamma\subset G$, ${\rm Cay}(\Gamma)$ and $C,D\in(0,\infty)$ be as in Lemma~\ref{lm:lattice}.

Fix $p_0\in(0,1]$. We will again show that there exists $\lambda_0>0$ such that
\begin{equation}\label{equ:LROsprinkledSymm}
\inf \Big\{ \mathbb P_{p_0+\delta}^{(\lambda)} \big( (x_1,y_1) \leftrightarrow (x_2,y_2) \big) \, : \, x_1,x_2\in X, y_1,y_2\in Y \Big\}>0,
\end{equation}
for every $\lambda<\lambda_0$ and $\delta>0$.

{\em Step 1} (Pairs of points on different $Y$-levels). Let $p^*$ be the threshold guaranteed by Theorem~\ref{theorem:threshold} for ${\rm Cay}(\Gamma)$. Since $p^*<1$, the triangle inequality and \eqref{equ:localuniquenessSymmSpace} imply that there exists $R>0$ and $\lambda_1>0$ such that for all $x\in X$ and $y,y'\in Y$,
$$
\inf_{\lambda\le\lambda_1} \mathbb P\big( A(\lambda,p_0,R,2R+D,(x,y)) \cap A(\lambda,p_0,R,2R+D,(x,y')) \big)>p^*.
$$
Note the choice of $S=2R+D$ in the local uniqueness condition here.

Fix an origin $o_X\in X$, fix $y_1,y_2\in Y$ and fix $R$ as above. For $\lambda>0$, $x\in X$, $y_1,y'\in Y$ and $N\ge1$, let 
$$
A(\lambda,p_0,R,(x,y),(x,y'),N)
$$
denote the event that all $\omega_{p_0}^{(\lambda)}$-clusters intersecting $B_R((x,y))$ and $B_R((x,y'))$ are at Delaunay distance at most $N$. An argument similar to the proof of \cite[Proposition~4.16]{GR25} shows that for every $\delta>0$, there exist $N\ge1$ and $\lambda_2>0$ such that
$$
\inf_{\lambda\le\lambda_2} \mathbb P (A(\lambda,p_0,R,(o_X,y_1),(o_X,y_2),N)) \ge 1-\delta
$$ 
(namely because the event is implied by all Voronoi cells intersecting the two balls being at uniformly bounded Delaunay distance). 
Combined with invariance it follows that there exist $N\ge1$ and $\lambda_0>0$ such that
\begin{align} 
&\inf_{\lambda\le\lambda_0} \mathbb P\big( A(\lambda,p_0,R,2R+D,(x,y_1)) \cap A(\lambda,p_0,R,2R+D,(x,y_2)) \nonumber \\
& \qquad \qquad \qquad \cap A(\lambda,p_0,R,(x,y_1),(x,y_2),N) \big)>p^*, \label{equ:aux perc marginals Symm}
\end{align}
for all $x\in X$.

%Fix an origin $o_X\in X$ and fix $y_1,y_2\in Y$. Let $p^*$ be the threshold guaranteed by Theorem~\ref{theorem:threshold} for ${\rm Cay}(\Gamma)$. Since $p^*<1$, the triangle inequality and \eqref{equ:localuniquenessSymmSpace} imply that there exists $R>0$ and $\lambda_0>0$ such that for all $x\in X$ and $y,y'\in Y$,
%$$
%\inf_{\lambda\le\lambda_0} \mathbb P\big( A(\lambda,p_0,R,3R+D,(x,y)) \cap A(\lambda,p_0,R,3R+D,(x,y')) \big)>p^*.
%$$
%Note the choice of $S=2R+D$ in the local uniqueness condition here. {\color{blue}new:} Since the Delaunay graph associated to the Poisson--Voronoi tessellation, i.e.~the graph with vertices given by the cells and edges between cells with non-empty border, is a.s.~locally finite (see \cite[Lemma 6.5]{GR25}), there exists $N\ge1$ {\color{red} add an argument why is this uniform in $\lambda$, or monotone or something; use Proposition 4.16 \cite{GR25}} such that 
%\begin{align} 
%&\inf_{\lambda\le\lambda_0} \mathbb P\big( A(\lambda,p_0,R,3R+D,(x,y)) \cap A(\lambda,p_0,R,3R+D,(x,y')) \nonumber \\
%& \qquad \qquad \qquad \cap A(\lambda,p_0,R,(x,y),(x,y'),N) \big)>p^*, \label{equ:aux perc marginals Symm}
%\end{align}
%where $A(\lambda,p_0,R,(x,y),(x,y'),N)$ denotes the event that all clusters intersecting $B_R((x,y))$ and $B_R((x,y'))$ are at Delaunay distance at most $N$.

Fix $\lambda_0$ and $N$ as above and let $\lambda<\lambda_0$. We define an auxiliary site percolation model $\xi$ on ${\rm Cay}(\Gamma)=(\Gamma,E)$ as follows: $g\in \Gamma$ belongs to $\xi$ if and only if the events $A(\lambda,p_0,R,2R+D,(go_X,y_1))$, $A(\lambda,p_0,R,2R+D,(go_X,y_2))$ and $A(\lambda,p_0,R,(go_X,y_1),(go_X,y_2),N)$ occur . Let $e\in\Gamma$ denote the identity. By \eqref{equ:aux perc marginals Symm} and Theorem~\ref{theorem:threshold}, we have that $\mathbb P(|C_\xi(e)|=\infty)\ge 2/3$. Thus with probability at least $2/3$, there exists an infinite path $\rho\coloneqq (g_1,g_2,g_3,\ldots)$ in ${\rm Cay}(\Gamma)$ starting at $g_1=e$ and such that the events $A(\lambda,p_0,R,2R+D,z)$ occur for all points $z\in \{ (g_io_X,y_1), (g_io_X,y_2) : i=1,2,\ldots\}$ and also the event $A(\lambda,p_0,R,z,z',N)$ occurs whenever $z= (g_io_X,y_1)$ and $z'=(g_io_X,y_2)$. In particular, the discrete paths 
\begin{align*}
& \rho \times y_1 \coloneqq  \big( (o_X,y_1), (g_2o_X,y_1), (g_3o_X,y_1),\ldots\big), \\
& \rho \times y_2 \coloneqq \big( (o_X,y_2), (g_2o_X,y_2), (g_3o_X,y_2),\ldots\big)
\end{align*}
in $X\times Y$ are such that within distance $R$ from each of their members is an $\omega_{p_0}^{(\lambda)}$-cluster, which moreover is unique within a ball of radius $3R+D$ centered at each path step. Since the distance between two consecutive path steps on the same $Y$-level is at most $D$, this implies that there are $\omega_{p_0}^{(\lambda)}$-clusters, one in the $R$-neighborhood of $\rho\times\{y_1\}$ and one in the $R$-neighborhood of $\rho\times\{y_2\}$, which are the unique clusters seen from members of the respective paths in the $3R+D$ neighborhood. It also follows from the definition and the fact that Voronoi cells are compact a.s.~(see \cite[Lemma 2.6]{GR25}) that these two clusters are at Delaunay distance at most~$N$ for infinitely many cells. Thus for every $\delta>0$
$$
\mathbb P_{p_0+\delta}^{(\lambda)} \big( B_R((o_X,y_1)) \leftrightarrow B_R((o_X,y_2)) \big) \ge 2/3.
$$
By invariance and the fact that the above bound did not depend on $y_1$ and $y_2$, we obtain that
\begin{equation} \label{equ:TPF balls Symm}
    \inf \Big\{ \mathbb P_{p_0+\delta}^{(\lambda)} \big( B_R((x,y)) \leftrightarrow B_R((x,y')) \big) : x\in \Gamma o_X, y,y'\in Y \Big\} \ge 2/3,
\end{equation}
which concludes the proof of Step 1.

{\em Step 2.} (General pairs). Fix $0<\lambda< \lambda_0$. Let $x_1,x_2\in \Gamma o_X$ and $y_1,y_2\in Y$. Choose an unbounded subset 
$$
Y_0\coloneqq \{y^{(1)},y^{(2)},\ldots\}\subset Y.
$$
For $y\in Y$ and $N\ge1$, let 
$$
H(y,\lambda,\delta,N)
$$ 
denote the event that in $\omega_{p_0+\delta}^{(\lambda)}$ we have that $B_R((x_1,y_1)) \leftrightarrow B_R((x_1,y))$ and $B_R((x_2,y_2)) \leftrightarrow B_R((x_2,y))$, and moreover the Delaunay distance between every pair of cells such that one touches $B_R((x_1,y))$ and one touches $B_R((x_2,y))$ is at most~$N$. 

By \eqref{equ:TPF balls Symm}, the fact that there are a.s.~only finitely many $\omega_{p_0+\delta}^{(\lambda)}$-clusters intersecting any fixed ball (see \cite[Lemma 2.3]{GR25}) and a.s.~local finiteness of the Delaunay graph (see \cite[Lemma 6.5]{GR25}), we obtain that
$$
\liminf_{N\to\infty} \mathbb P(H(y,\lambda,\delta,N))\ge 1/3.
$$
Choose $N_0\ge1$ such that $\mathbb P(H(y,\lambda,\delta,N_0))\ge1/4$. Then for every $\delta>0$
$$
\mathbb P \big( H(y^{(i)},\lambda,\delta,N_0) \text{ for infinitely many } y^{(i)} \big) \ge 1/4.
$$
Using again that there are a.s.~only finitely many $\omega_{p_0+\delta}^{(\lambda)}$-clusters intersecting any fixed ball and a.s.~compactness of Poisson--Voronoi cells, we conclude that on the event
$$
H\coloneqq  \big\{  H(y^{(i)},\lambda,\delta,N_0) \text{ for infinitely many } y^{(i)} \big\},
$$
there exist two $\omega_{p_0+\delta}^{(\lambda)}$ clusters, one intersecting $B_R((x_1,y_1))$ and one intersecting $B_R((x_2,y_2))$, whose Delaunay distance is at most $N_0$ for infinitely many cells. Thus for every $\delta>0$ 
$$
\inf \Big\{ \mathbb P_{p_0+\delta}^{(\lambda)} \big( B_R((x,y)) \leftrightarrow B_R((x',y')) \big) \, : \, x,x'\in \Gamma o_X, y,y'\in Y \Big\} \ge 1/4.
$$
Since $\Gamma o_X \times Y\subset M$ is $C$-dense, the FKG-inequality (cf.~\cite[Lemma 2.5]{GR25}) together with invariance imply \eqref{equ:LROsprinkledSymm}. The proof of Theorem~\ref{thm:LocalToGlobalSymm} is thus complete.
\end{proof}

\subsubsection{Poisson--Voronoi percolation on general products}

Theorem~\ref{thm:LocalToGlobalSymm} suffices for all concrete applications to continuum percolation in this paper. We may also point out the following broad generalization, which follows immediately from the above arguments and which might be of independent interest: Consider the following assumptions.
\begin{itemize}
\item $X$ and $Y$ are non-compact proper geodesic lcsc metric spaces equipped with infinite Radon measures $\mu_X$ and $\mu_Y$ respectively;
\item $G$ and $H$ are lcsc groups acting properly, continuously, transitively and by measure-preserving isometries on $X$ and $Y$ respectively. Moreover, $G$ is non-amenable and admits a co-compact lattice;
\item The direct product $M\coloneqq X\times Y$ is equipped with the measure $\mu_X\otimes\mu_Y$ and $L^p$-distance $d_M((x,y),(x',y'))\coloneqq (d_X(x,x')^p+d_Y(y,y')^p)^{1/p}$ for some $p\ge1$;
\item Poisson--Voronoi percolation on $M$ has local uniqueness at low intensities in the sense of \eqref{equ:localuniquenessSymmSpace} and a.s.~has compact cells each of which has finitely many neighbors.\footnote{The fact that every ball intersects at most finitely many cells, also used in the proof, follows from the assumptions.}
\end{itemize}

\begin{proposition}
Under the above assumptions, Poisson-Voronoi percolation on $M$ satisfies that for every $p_0\in(0,1]$, there exists $\lambda_0>0$ such that 
$$
\inf_{x,y\in M} \mathbb P_{p_0}^{(\lambda)}(x\leftrightarrow y) >0
$$
for all $\lambda<\lambda_0$.
\end{proposition}
\begin{proof} The proof is the same as that of Theorem~\ref{thm:LocalToGlobalSymm}.
\end{proof}

\section{Local uniqueness for products of trees}\label{sec:Trees}

\newcommand{\dist}{\operatorname{dist}}

Throughout this section, fix $d\ge 3$, $k\ge 2$, $G\coloneqq (\mathbb{T}_d)^k$ as the $k$-fold product of $\mathbb{T}_d$, and ${\bf o}=(o,\dots,o)\in V(G)$ a fixed root.

For $\lambda\in(0,1]$, $p\in(0,1]$, $R\ge1$ and $x\in V(G)$, recall that we denote by $A(\lambda,p,R,x)$ the event that $\omega_p^{(\lambda)}$ intersects $B_R(x)$ with at least one cluster and $\omega_p^{(\lambda)}\cap B_{3R}(x)$ is contained in a single cluster of $\omega_p^{(\lambda)}$.

\begin{theorem}[\textsc{Local uniqueness}] \label{thm:localuniqueness}
Let $G$ be the $k$-fold product of $d$-regular trees and let $x\in V(G)$.
Then for every $p_0\in(0,1]$ and $\varepsilon>0$, there exist $R>0$ and $\lambda_0\in(0,1)$ such that
\begin{equation}\label{equ:finitarytouching}
\inf_{\lambda\le\lambda_0} \mathbb P\big( A(\lambda,p_0,R,x) \big)>1-\varepsilon.
\end{equation}
\end{theorem}

This section is devoted to the proof of Theorem~\ref{thm:localuniqueness}.
The similarities and differences between this section and \cite{FMW23} are explained in detail in Section \ref{sec:introdiscrete}.
On a high level, the proof is as follows.
We encode a Bernoulli--Voronoi diagram as a bond percolation process by erasing all edges between vertices in different cells, and we show that any sequence of intensities $\lambda_n\to 0$ admits a subsequence $\lambda_{n_k}\to 0$ such that the Bernoulli--Voronoi diagrams encoded as bond percolations converge weakly to a limiting bond percolation, and this limiting process encodes the Voronoi diagram associated to a Poisson point process $\{X_i\}_{i\in\mathbb N}$ on a suitable boundary of~$G$.
The assumption that all factors in $G$ are $\mathbb{T}_d$ allows us to interpret $X_i:G\to \mathbb{Z}$ as a sum of horofunctions up to an additive constant, which we use to prove that almost all pairs of cells in the ideal Poisson Voronoi tessellation, i.e.~the diagram determined by $\{X_i\}_{i\in \mathbb{N}}$, share unbounded borders.
We further show that in this ideal tessellation, large balls intersect many cells with high probability. 
Using these two facts, fixing $p_0\in (0,1]$ and taking $k\in \mathbb{N}$ large enough guarantees
$$
\mathbb P\big( A(\lambda_{n_k},p_0,R,x) \big)>1-\varepsilon,
$$
as desired.

\subsection{The space of distance-like functions} \label{sec:Corona}

We start with fixing the notation.
We use the convention that $V(G)=G$ and denote as $\dist({-},{-})$ the graph distance on $G$.
That is, we have
$$\dist(a,b)=\sum_{i=1}^k \dist(a_i,b_i)$$
for every $a=(a_i)_{i=1}^k\in (\mathbb{T}_d)^k$, $b=(b_i)_{i=1}^k\in (\mathbb{T}_d)^k$, where with an abuse of notation we denote by $\dist({-},{-})$ the graph distance on $\mathbb{T}_d$ as well.
Let $\partial\mathbb{T}_d$ denote the totally disconnected and uncountable space of all ends in $\mathbb{T}_d$ and $\overline{\mathbb{T}_d}\coloneqq \mathbb{T}_d\cup \partial\mathbb{T}_d$. We endow both $\partial\mathbb{T}_d$ and $\overline{\mathbb{T}_d}$ with their canonical topologies. 
We say that $r\in \mathbb{T}_d$ \textbf{separates} two subsets $A$ and $B$ of $\overline{\mathbb{T}_d}$ if any geodesic ray in $\mathbb{T}_d$ 
intersecting both $A$ and $B$ must also contain $r$.
%connecting any element of $A$ to any element of $B$ has to go through $r$.

Given $\psi\in \overline{\mathbb{T}_d}$, we define the corresponding (normalized) horocycle $h^\psi:\mathbb{T}_d\to \mathbb{Z}$ as
$$h^\psi(a)=\dist(a,\psi)-\dist(o,\psi)$$
whenever $\psi\in \mathbb{T}_d$, and
$$h^\psi(a)=\dist(a,r)-\dist(o,r)$$
for any $r\in \mathbb{T}_d$ that separates $\{\psi\}$ from $\{a,o\}$ in the case $\psi\in \partial\mathbb{T}_d$.
Given a vector $\psi=(\psi_1,\dots,\psi_k)\in (\overline{\mathbb{T}_d})^k$, we define the corresponding (normalized) horocycle $h^\psi:G\to \mathbb{Z}$ as
$$h^\psi(a)=\sum_{i=1}^k h^{\psi_i}(a_i)$$
for $a=(a_1,\ldots,a_k)\in G$.
Note that $h^\psi({\bf o})=0$.

We write ${\bf D}$ for the {\bf space of distance-like functions} which is defined to be the closure in~$\mathbb{Z}^G$, under pointwise convergence, of the following set of shifted distance functions
$$\{\dist(a,{-})+\ell:a\in G, \ell\in \mathbb{Z}\}.$$
As a topological space, that is, endowed with the topology of pointwise convergence, ${\bf D}$ is a locally compact second countable (lcsc) space.
While it follows directly from the definition that every $f\in {\bf D}$ is $1$-Lipschitz, Proposition \ref{pr:BasicCorona} shows that every element of ${\bf D}$ is equal to a horofunction up to an additive constant.

Denote the automorphism group of $\mathbb T_d$ as $\aut(\mathbb T_d)$ and set $H\coloneqq (\aut(\mathbb T_d))^k$.
The group $H$ acts naturally on the space of distance-like functions $\mathbf D$.
Namely, if $\tau\in H,f\in \mathbf D$ and $v=(v_1,\ldots,v_k)\in G$, then we have
\[
\tau f(v)=f(\tau^{-1}v).
\]
The group $H$ also acts canonically on $(\partial \mathbb{T}_d)^k$ by permuting the ends in each coordinate.
The action of $H$ on $(\partial \mathbb{T}_d)^k$ extends to $(\partial \mathbb{T}_d)^k\times \mathbb{Z}$ by 
\[
\tau (\psi,m)=(\tau(\psi),h^\psi(\tau^{-1}({\bf o}))+m)
\]
for every $\tau\in H$ and $(\psi,m)\in (\partial \mathbb{T}_d)^k\times \mathbb{Z}$. We note that this is a discrete analogue of the Maharam extension of $(\partial \mathbb{T}_d)^k$ (see e.g. \cite[Section 4.2]{BN2013}).

Justified by Proposition~\ref{pr:BasicCorona}, we view $(\partial \mathbb{T}_d)^k\times \mathbb{Z}$ as a subset of ${\bf D}$ and use $(\psi,m)$ and $h^\psi+m$ interchangeably.
In fact, the map $(\psi,m)\to h^\psi+m$ defines a bijection between $(\overline{\mathbb{T}_d})^k\times \mathbb Z$ and $\bf D$.
As the proof of Proposition~\ref{pr:BasicCorona} is routine, we only include it in Appendix~\ref{app:Trees}. 
See Lemmas 5.4 and 5.5 in \cite{FMW23} for the relevant analogues to Proposition \ref{pr:BasicCorona} and Theorem \ref{thm:SubseqLimit}, below.

\begin{proposition}\label{pr:BasicCorona}
    Let $d\ge 3$, $k\ge 2$, $G=(\mathbb{T}_d)^k$ and ${\bf D}$ be defined as above.
    \begin{enumerate}
        \item[{\rm(i)}] For every $f\in {\bf D}$, there exist $\psi=(\psi_1,\dots,\psi_k)\in (\overline{\mathbb{T}_d})^k$ and $m\in \mathbb{Z}$ such that $f=h^\psi+m$.
        \item[{\rm(ii)}] For every $\psi=(\psi_1,\dots,\psi_k)\in (\overline{\mathbb{T}_d})^k$ and $m\in \mathbb{Z}$, we have that $h^\psi+m\in {\bf D}$.
        \item[{\rm(iii)}] The map $(\psi,m)\mapsto h^\psi +m$ from $(\partial \mathbb{T}_d)^k\times \mathbb{Z}$ to ${\bf D}$ is injective, continuous, has closed range and is $H$-equivariant.
    \end{enumerate}
\end{proposition}

\subsection{Invariant measures on the space of distance-like functions}

\newcommand{\vol}{\operatorname{vol}}

In this section we construct an $H$-invariant measure on the space of distance-like functions $\mathbf D$.
We denote by $\nu$ counting measure on $G$.
Observe that $\nu$ is $H$-invariant.

\begin{definition}
    For $t\in \mathbb{Z}$, 
    we define the $H$-equivariant embeddings $\iota_t:G\to {\bf D}$ as
    $$\iota_t((a_i)_{i=1}^k)=\dist(-,a)-t=\sum_{i=1}^k\dist(-,a_i)-t,$$
    where $a=(a_i)_{i=1}^k\in (\mathbb{T}_d)^k$.
\end{definition}

Given $\lambda\in (0,1]$ and $t\in\mathbb Z$, we denote by $\iota_t^*(\lambda\cdot\nu)$ the push-forward of $\lambda\cdot \nu$ via~$\iota_t$.
Note that $\iota_t^*(\lambda\cdot\nu)$ is $H$-invariant for any choice of $t\in \mathbb{Z}$ and $\lambda\in (0,1]$.
In the setting of higher rank symmetric spaces, see \cite[Theorem~2.2~(5)]{GR25}, we know that for every $\lambda>0$ there is $t\in \mathbb{R}$ such that $\vol(B_t(o))=\lambda^{-1}$, where $\vol$ is the corresponding Haar measure.
This is clearly not true in the discrete case.

\begin{definition}
    For $\lambda\in (0,1]$, define $t_\lambda$ to be the maximal $\mathbb Z$-value such that $\nu(B_{t_{\lambda}}(o))\le \lambda^{-1}$ and set $\mu_{\lambda}\coloneqq \iota_{t_\lambda}^*(\lambda\cdot\nu)$.
\end{definition}

By the definition we have that $\nu(B_{t_\lambda}(o))\le \lambda^{-1} \le \nu( B_{t_\lambda+1}(o))\le kd\cdot \nu(B_{t_\lambda}(o))$.
This can be compactly stated as $\lambda\in \Theta(\nu( B_{t_\lambda}(o))^{-1})$ for $\lambda\to 0$, which will be used repeatedly in the proofs.

Recall that a sequence of Radon measures $\mu_n$ on ${\bf D}$ \textbf{converges vaguely} to a measure~$\mu$, in symbols $\mu_n \xrightarrow{v} \mu$, if
$$\int_{{\bf D}} F \ d\mu_n\to \int_{{\bf D}} F \ d\mu$$
for every continuous $F:{\bf D}\to \mathbb{R}$ with compact support.
We write $\mathbb{M}(\bf D)$ for the space of Radon measures on ${\bf D}$ endowed with the topology of vague convergence.

\begin{theorem}[\textsc{Subsequential limits}]\label{thm:SubseqLimit}
    %Let $d\ge 3$, $k\ge 2$ and $G=(\mathbb{T}_d)^k$ be the $k$-fold product of $\mathbb{T}_d$ with the graph product metric.
    The sequence $(\mu_\lambda)_{0<\lambda\le 1}\subseteq \mathbb{M}({\bf D})$ is relatively compact.
    Moreover, any subsequential limit $\mu_{\lambda_n}\to \mu\in \mathbb{M}({\bf D})$ as $\lambda_n\to 0$ is:
    \begin{enumerate}
        \item[{\rm(i)}] $H$-invariant,
        \item[{\rm(ii)}] supported on $(\partial \mathbb T_d)^k\times \mathbb Z$,
        \item[{\rm(iii)}] of the form $(\beta)^k\times \theta$, where $\beta$ is the harmonic measure on $\partial\mathbb{T}_d$ and $\theta$ is a non-zero measure on $\mathbb{Z}$ such that $\theta(m)=\theta(0)(d-1)^m$.
    \end{enumerate}
\end{theorem}

\begin{proof}
    The fact that the sequence $(\mu_\lambda)_{0<\lambda\le 1}\subseteq \mathbb{M}({\bf D})$ is relatively compact follows as in \cite[Proposition 3.3, Lemma 3.4]{FMW23}; 
    alternatively one may directly use the estimates below, using the fact that for fixed $p\in\mathbb N$,
    \begin{align}\label{eq:Growth}
        & |\{w\in (\mathbb{T}_d)^p:\dist({\bf o},w)=q\}|,\nonumber\\
        & |\{w\in (\mathbb{T}_d)^p:\dist({\bf o},w)\le q\}|\in \Theta\left(q^{p-1} (d-1)^q\right)
    \end{align}
    as $q\to\infty$. 
    
    For completeness, we include a proof of \eqref{eq:Growth}. It suffices to prove the statement for $|\{w\in (\mathbb{T}_d)^p:\dist({\bf o},w)=q\}|$. Note that
    \begin{align*}
        & |\{w\in (\mathbb{T}_d)^p:\dist({\bf o},w)=q\}| \\
        & \qquad = \sum_{q_1,\ldots,q_p} |\{w\in (\mathbb{T}_d)^p: \dist(o,w_1)=q_1,\ldots,\dist(o,w_p)=q_p\}|,
    \end{align*}
    where the sum runs over $q_1,\ldots,q_p$ with $q_i\in\{0,\ldots,q\}$ such that $\sum_{i=1}^p q_i=q$. For every summand, we have that 
    $$
    |\{w\in (\mathbb{T}_d)^p: \dist(o,w_1)=q_1,\ldots,\dist(o,w_p)=q_p\}| \le \prod_{i=1}^p d(d-1)^{q_i-1} \in \Theta((d-1)^q)
    $$
    (note that the inequality is an equality except when some $q_i$ is zero) and
     $$
    |\{w\in (\mathbb{T}_d)^p: \dist(o,w_1)=q_1,\ldots,\dist(o,w_p)=q_p\}| \ge \prod_{i=1}^p (d-1)^{q_i}=(d-1)^q,
    $$
    while the number of summands is $\binom{q+p-1}{p-1}\in\Theta(q^{p-1})$. The claim follows. 
    
    Item (i) follows as $\mu_\lambda$ is $H$-invariant for every $\lambda$ combined with the fact that $H$-invariance is preserved by passing to vague limits.

    (ii).
    Suppose for a contradiction that
    $$\mu({\bf D}\setminus ((\partial \mathbb T_d)^k\times \mathbb Z))>0.$$ 
    By Proposition~\ref{pr:BasicCorona} (i), $H$-invariance and $\sigma$-additivity of $\mu$, permuting the coordinates if necessary, there are $m\in \mathbb{Z}$ and $\ell\in\{1,\ldots,k\}$ such that $\mu(X_{m,\ell})>0$, where
    $$X_{m,\ell}=\{h^\psi+m:\psi\in \{o\}^\ell\times (\overline{ \mathbb{T}_d})^{k-\ell}\}.$$
    Observe that $X_{m,\ell}$ is a compact and open subset of ${\bf D}$. It follows that 
    $$\lambda_n\cdot \nu(\iota_{t_{\lambda_n}}^{-1}(X_{m,\ell}))=\mu_{\lambda_n}(X_{m,\ell})\to \mu(X_{m,\ell}).$$
    Set $Y_{m,\ell,n}\coloneqq  \iota_{t_{\lambda_n}}^{-1}(X_{m,\ell})$. The definition of $\iota_{t_{\lambda_n}}$ then gives that $Y_{m,\ell,n}$ is equal to
    \begin{equation}\label{eq:Definition}
        \begin{split}
            \left\{(v_1,\dots,v_k)\in (\mathbb{T}_d)^k:v_1=\dots=v_\ell=o, \ \sum_{i=\ell+1}^k\dist(o,v_i)=m+t_{\lambda_n}\right\}.
        \end{split}
    \end{equation}
    By \eqref{eq:Growth}, we have that $\lambda\in \Theta(\nu( B_{t_\lambda}(o))^{-1})=\Theta(t_\lambda^{-k+1}(d-1)^{-t_\lambda})$ as $\lambda\to0$. Similarly,
    $$
    \nu(Y_{m,\ell,n}) \in \Theta( (m+t_{\lambda_n})^{k-\ell-1}(d-1)^{m+t_{\lambda_n}})
    $$
    as $n\to\infty$. Combining the two, we obtain that
    $$
    \lambda_n \cdot \nu(Y_{m,\ell,n}) \in \Theta\bigg( \frac{(m+t_{\lambda_n})^{k-\ell-1}(d-1)^{m+t_{\lambda_n}}}{t_{\lambda_n}^{k-1}(d-1)^{t_\lambda}} \bigg) = \Theta( t_{\lambda_n}^{-\ell})=o(1)
    $$
    as $n\to\infty$, which contradicts the assumption $\mu(X_{m,\ell})>0$.

    (iii). Define 
    \[\theta(m)=\mu((\partial \mathbb{T}_d)^k\times \{m\})\overset{\rm (ii)}{=}\mu((\overline{\mathbb{T}_d})^k\times \{m\})=\mu(X_{m,0}).\]
    We show that $0<\theta(m)<\infty$.
    As $X_{m,0}\subset\mathbf D$ is compact and open, we have that
    $$
    \mu(X_{m,0}) = \lim_{n\to\infty} \lambda_n \nu(Y_{m,0,n}).
    $$
    Recall from the proof of \eqref{eq:Growth} that
    \begin{align*}
        \nu(Y_{m,0,n}) & \le \binom{m+t_{\lambda_n}+k-1}{k-1}d^k(d-1)^{m+t_{\lambda_n}-k} \\
        & = \frac{d^k}{(d-1)^k(k-1)!} (m+t_{\lambda_n}+k-1)\cdot\ldots\cdot(m+t_{\lambda_n}+1)(d-1)^{m+t_{\lambda_n}} 
    \end{align*}
    and
    \begin{align*}
        \nu(Y_{m,0,n}) & \ge \binom{m+t_{\lambda_n}+k-1}{k-1}(d-1)^{m+t_{\lambda_n}} \\
        & = \frac{1}{(k-1)!} (m+t_{\lambda_n}+k-1)\cdot\ldots\cdot(m+t_{\lambda_n}+1)(d-1)^{m+t_{\lambda_n}}. 
    \end{align*}
    Since $\lambda_n=\Theta(t_{\lambda_n}^{-k+1}(d-1)^{-t_{\lambda_n}})$ as $n\to\infty$, we obtain that there exist constants $c,C>0$, which depend on $k$ and $d$ but not on $m$, such that 
    $$
    c \cdot (d-1)^m \le \lim_{n\to\infty} \lambda_n \nu(Y_{m,0,n}) \le C(d-1)^m.
    $$
    In particular, we have that $0<\theta(m)<\infty$ for every $m\in \mathbb{Z}$.
    
    Let $H_{\bf o}$ be the stabilizer of ${\bf o}$ in $H$.
    Observe that if $\tau\in H_{\bf o}$, $(\psi,m)\in (\partial \mathbb{T}_d)^k\times \mathbb{Z}$, then $\tau\cdot (\psi,m)=(\tau(\psi),m)$.
    That is, $(\partial \mathbb{T}_d)^k\times \{m\}$ is $H_{\bf o}$-invariant for every $m\in \mathbb{Z}$.
    It follows that for every $m\in \mathbb{Z}$ we have that $\mu\upharpoonright \left((\partial \mathbb{T}_d)^k\times \{m\}\right)=\theta(m)\beta^k$, where $\beta$ is the harmonic measure on $\partial\mathbb{T}_d$.
    
    Finally, we show that for every $m\in \mathbb{Z}$, we have that $\frac{\theta(m+1)}{\theta(m)}=d-1$.
    Pick $v\in \mathbb{T}_d$ to be an arbitrary neighbor of $o$ in the $1$st copy of $\mathbb{T}_d$ and define
    $$A=\left\{(\psi_1,\dots,\psi_k)\in (\partial \mathbb{T}_d)^k: v\in [0,\psi_1]\right\},$$
    where $[0,\psi]$ denotes the geodesic from $o$ to $\psi$.
    Observe that $A$ is measurable and $\mu(A\times \{m\})=\frac{\theta(m)}{d}$ for every $m\in \mathbb{Z}$.
    Let $\alpha=(\alpha_1,\dots,\alpha_k)\in H$ be such that $\alpha_2=\dots=\alpha_k=\id_{\mathbb{T}_d}$, and $v\in [o,\alpha_1(\psi)]$ whenever $v\in [o,\psi]$.
    In other words, $\alpha$ shifts the ends in the first coordinate towards $v$ and leaves the other coordinates intact.
    By $H$-invariance of $\mu$, we have that
    $$\mu(\alpha\cdot (A\times \{m\}))=\mu(A\times \{m\})=\frac{\theta(m)}{d}.$$
    By the definition, we have that if $(\psi,m)=((\psi_1,\dots,\psi_k),m)\in A\times \{m\}$, then
    \begin{equation*}
        \begin{split}
            \alpha\cdot (\psi,m)= & \ ((\alpha(\psi),h^\psi(\alpha^{-1}({\bf o}))+m)\\
            = & \ ((\alpha_1(\psi_1),\psi_2,\dots,\psi_k),m+1).
        \end{split}
    \end{equation*}
    By choice of $\alpha_1$, we have that $\mu(\alpha\cdot (A\times \{m\}))=\frac{\theta(m+1)}{d(d-1)}$. The claim follows.  
\end{proof}

\begin{remark}
%\newMD{In the upcoming paper~\cite{DK24}, the first and third author prove (among other things) that for any graph having nice growth behavior, fixing $\xi\in \mathbb R$ and choosing $\lambda \coloneqq  \lambda_t\coloneqq  \nu(B_{t(o)})^{-1} p^\xi$ (where $p$ is the growth exponent ($p=d-1$ for $\mathbb{T}_d$, and $t=1,2,\ldots$)), any subsequential limit of the measure $\iota^*_t(\lambda\cdot\mu)$ satisfies $\forall m: \theta(m)=(p-1)p^{m-1+\xi}$. Also, for graphs with even nicer growth (which include $\mathbb{T}_d^k$), the full sequence $\iota^*_t(\lambda_t\cdot\mu)$ converges when $\xi$ is fixed.} 
In the upcoming paper \cite{DK24}, the first and third author (among other things) strengthen Theorem~\ref{thm:SubseqLimit} by showing that fixing $\xi\in \mathbb R$ and choosing $\lambda_t\coloneqq  \nu(B_{t}(o))^{-1} (d-1)^\xi$, the measures $\iota^*_t(\lambda_t\cdot\mu)$ converge to a limit satisfying $\theta(m)=(d-2)(d-1)^{m-1+\xi}$. Moreover, this is generalized in \cite{DK24} to graphs with sufficiently nice growth.
\end{remark}

\subsection{Unbounded borders in products}

The next definition appears in a similar form in \cite[Section~6]{FMW23}. Let $R>1$, $\Pi$ be a countable subset of $\bf D$, and $f_1,f_2\in \Pi$. The {\bf $R$-wall} of~$\Pi$ with respect to $(f_1,f_2)$ is 
\begin{align*}
    & W(R,\Pi,f_1,f_2)\coloneqq \\
    & \qquad \{v\in G: |f_1(v)-f_2(v)|\le 1 \text{ and } \forall f\in \Pi\setminus \{f_1,f_2\},  f_1(v)+R< f(v)\}.
\end{align*}
Note that while \cite[Section~6]{FMW23} requires $f_1(v)=f_2(v)$ in their definition in the continuous setup, in the discrete case this has to be replaced by $|f_1(v)-f_2(v)|\le 1$.
%\textcolor{orange}{doesn't this depend on the choice of metric? regardless, can you specify what fails with the L1 metric?}
%{\color{red} if for example two points $X^\lambda_i,X^\lambda_j$ in $\mathbb{T}_d$ are of odd distance, then there is no $v\in \mathbb{T}_d$ such that $\dist(X_i^\lambda,v)=\dist(X_j^\lambda,v)$, as the middle point between them is in the middle of some edge.}

Throughout this section we let $\mu$ be a subsequential vague limit of $(\mu_{\lambda_n})$ as $\lambda_n\rightarrow 0$.

\begin{theorem}[{\textsc{Unbounded borders}}]\label{thm:UnboundedTouching}
%Suppose that $\mu_{\lambda_{n}}\xrightarrow{v}\mu$, where $\lambda_n\to 0$.
The Poisson point process $\bf X$ with intensity $\mu$ on $\bf D$ has the following property almost surely. For every $R>1$ and $f_1,f_2\in \mathbf X$, the $R$-wall $W(R,\mathbf{X},f_1,f_2)$ is infinite.
\end{theorem}

The proof of Theorem~\ref{thm:UnboundedTouching} is done in the same way as the proof of \cite[Theorem~6.1]{FMW23}.
We include here completely elementary proofs for the combinatorial core of the argument, Proposition~\ref{prop:UnboundedStab} and Lemma~\ref{lm:HoweMoore}.
We give a sketch for the proof of Theorem \ref{thm:UnboundedTouching} at the end of this section.

\begin{proposition}[\textsc{Unbounded stabilizers}] \label{prop:UnboundedStab} 
    %Let $G=(\mathbb T_d)^k$, $H=(\aut(\mathbb T_d))^k$ and $\mathbf D$ be as defined above.
    Let $f,g\in {\bf D}$ with $f=h^\psi+\ell$ and $g=h^\phi+\ell'$, where $\psi=(\psi_1,\ldots,\psi_k)\in(\partial \mathbb T_d)^k$ and $\phi=(\phi_1,\ldots,\phi_k)\in(\partial \mathbb T_d)^k$ such that 
    there exist $1\leq i,j\leq k,i\neq j$ with $\psi_i\neq\phi_i$ and $\psi_j\neq\phi_j$.
    %$\psi_i\neq\phi_i$ for all $1\leq i\leq k$. 
    Then there exists an unbounded sequence in $H$ fixing both $f$ and $g$.
\end{proposition}

\begin{proof}
    Let $1\leq i,j\leq k$ with $\psi_i\neq\phi_i$ and $\psi_j\neq\phi_j$. We denote the the unique geodesic from $\psi_i$ to $\phi_i$ as $[\psi_i,\phi_i]\subset \overline{\mathbb T_d}$ and use similar notation for any geodesic or geodesic ray in $\overline{\mathbb T_d}$.
    Pick $\sigma_i\in\aut(\mathbb T_d)$ such that if $v_i\in[\psi_i,\phi_i],$ then $\sigma_i(v_i)$ is the neighbor of $v_i$ in $[\psi_i,\phi_i]$ closer to $\phi_i$, and pick $\sigma_j\in\aut(\mathbb T_d)$ the same way.

    Let $N\coloneqq \{(n_1,n_2)\in\mathbb Z^2 : n_1> 0,n_2< 0,n_1+n_2=0\}$, and let $n\coloneqq (n_1,n_2)\in N$.
    Set $\tau_i=\sigma_i,\tau_j=\sigma_j$ and
    \[\tau_n\coloneqq (\id,\ldots,\tau_i^{n_1},\ldots,\tau_j^{n_2},\ldots,\id).\]
    Then $\tau_n\in H$ and moves any vertex in $G$ component-wise in the following way:  component $i$ moves closer to $\phi_i$ by $n_1$ steps, component $j$ moves closer to $\psi_j$ by $n_2$ steps, and all other components are fixed.
    We show that $\tau_n$ stabilizes both $f=h^{\psi}+m$ and $g=h^{\phi}+\ell'$.

    To do so we show that $f(v)=\tau_n f(v)$ for every $v\in G$; the argument for $g$ is completely analogous.
    Let $v=(v_1,\ldots, v_k)\in G$.
    Then, by the definition, we have
    \begin{equation*}
    \begin{split}
        \tau_n f(v)-f(v) = h^{\psi_i}(\sigma_i^{-n_1}(v_i))-h^{\psi_i}(v_i) + h^{\psi_j}(\sigma_j^{-n_2}(v_j))-h^{\psi_j}(v_j).
    \end{split}
    \end{equation*}
    We claim that $-n_1=h^{\psi_i}(\sigma_i^{-n_1}(v_i))-h^{\psi_i}(v_i)$ and $-n_2=h^{\psi_j}(\sigma_j^{-n_2}(v_j))-h^{\psi_j}(v_j)$, which clearly implies the statement.

    Let $r_1\in [o,\psi_i]$ separate $\{\psi_i\}$ from $\{o,v_i\}$.
    Note that $\{\sigma_i^{-n_1}(r_1)\}$ separates $\{\psi_i\}$ from $\{o,\sigma_i^{-n_1}(v_i)\}$ as $\sigma^{-n_1}_i(o)\in [0,\psi_i]$ by $n_1>0$.
    By the definition, we can write
    \begin{equation*}
        \begin{split}
            h^{\psi_i}(\sigma_i^{-n_1}(v_i))=& \ \dist(\sigma_i^{-n_1}(v_i),\sigma_i^{-n_1}(r_1))-\dist(o,\sigma_i^{-n_1}(r_1))\\
            = & \ \dist(v_i,r_1)-\dist(\sigma_i^{n_1}(o),r_1)\\
            = & \ \dist(v_i,r_1)-(\dist(o,r_1)+n_1)=h^{\psi_i}(v_i)-n_1,
        \end{split}
    \end{equation*}
    where $\dist(\sigma_i^{n_1}(o),r_1)=\dist(o,r_1)+n_1$ follows from $o\in [r_1,\sigma_i^{n_1}(o)]$.
    To show that $-n_2=h^{\psi_j}(\sigma_j^{-n_2}v_j)-h^{\psi_j}(v_j)$, we start by taking $r_2\in [o,\psi_j]$ that separates $\{\psi_j\}$ from $\{o,\sigma^{-n_2}_j(v_j)\}$.
    After that the argument is analogous.
    This finishes the proof.
\end{proof}

\begin{lemma}\label{lm:HoweMoore}
    Let $d\ge3,k\ge2,G=(\mathbb T_d)^k$ and $\mathbf D$ be defined as above.
    Suppose that $\mu_{\lambda_{n}}\xrightarrow{v}\mu$ as $\lambda_n\to 0$.
    Then for every $m\in\mathbb Z$,
    \begin{align*}
        \mu\big( \{(\psi,\ell) \in (\partial\mathbb T_d)^k  \times\mathbb Z : \max \{ h^\psi(u)+\ell,& \  h^\psi(v)+\ell \} \le m \} \big) \\
        & \qquad \in O\Big( {\rm dist}(u,v)^k (d-1)^{-{\rm dist}(u,v)/4} \Big)
    \end{align*}
    as ${\rm dist}(u,v)\to\infty$.
\end{lemma}

\begin{proof}
    By $H$-invariance of $\mu$, we may assume that $u=\bf o$.
    Write $v=(v_1,\dots,v_k)$, $\|v_i\|=\dist(v_i,o)$ for every $1\le i\le k$ and $\|v\|=\sum_{i=1}^k\|v_i\|=\dist({\bf o},v)$.
    For every $0\le j\le \|v_i\|$ define
    $$A_{i,j}=\{\psi\in \partial \mathbb{T}_d:h^\psi(v_i)=\|v_i\|-2j\}.$$
    Observe that $\beta(A_{i,j})\in \Theta\left((d-1)^{-j}\right)$ as $j\to\infty$ uniformly in $i$, where $\beta$ is the harmonic measure on $\partial \mathbb{T}_d$.

    Given a sequence $J=(j_1,\dots,j_k)$ such that $0\le j_i\le \|v_i\|$, let
    $$m_J:= \min \left\{m, m -\Vert v \Vert + 2\sum_{i=1}^k j_i\right\}.$$
    By Theorem~\ref{thm:SubseqLimit}~(iii), we have that
    \begin{equation}\label{eq:Measure}
    \mu\left(\left(\prod_{i=1}^k A_{i,j_i}\right)\times \{\ell:\ell\le m_J\}\right)\in O\left((d-1)^{m_J-\sum_{i=1}^k j_i}\right)
    \end{equation}
    as $-\left(m_J-\sum_{i=1}^k j_i\right)\to\infty$.
    We claim that the right-hand side of \eqref{eq:Measure} can be replaced by $O((d-1)^{-\Vert v\Vert/4})$ as $\Vert v \Vert\to\infty$. 
    Indeed, either $\sum_{i=1}^k j_i\ge \|v\|/4$, or $\sum_{i=1}^k j_i< \|v\|/4$ in which case we have $m_J\le -\|v\|/2+m$.
    
    Finally, observe that for every $(\psi,p)\in (\partial\mathbb{T}_d)^k\times \mathbb{Z}$ that satisfies
    $$h^\psi({\bf o})+p,  \  h^\psi(v)+p\le m$$
    there is a unique sequence $J=(j_1,\dots,j_k)$ such that $0\le j_i\le \|v_i\|$ and $(\psi,p)\in \left(\prod_{i=1}^k A_{i,j_i}\right)\times \{\ell:\ell\le m_J\}$.
    The claim follows from the estimate above together with the fact that there are at most $(\|v\|+1)^k$ such sequences $J$.
\end{proof}

\begin{proof}[Sketch of the Proof of Theorem~\ref{thm:UnboundedTouching}]
    Proposition~\ref{prop:UnboundedStab} guarantees that two points chosing according to $\mu\times \mu$ in ${\bf D}$ have unbounded stabilizer under the action of $H$.
    %In fact, the stabilizer admits an explicit description.
    Lemma~\ref{lm:HoweMoore} is a special case of the Howe--Moore theorem; it states that for fixed $m\in \mathbb{Z}$ the event that the values of $f\in {\bf D}$ at $v,w\in G$ are bounded by $m$ decays exponentially in $\dist(v,w)$. The combination of these two results implies Theorem~\ref{thm:UnboundedTouching} via applications of the Mecke formula and Kochen--Stone theorem, as in \cite[Section~6]{FMW23}.
\end{proof}

%\textcolor{orange}{I can add the rest of this argument. Kochen-Stone is actually not necessary, I can use a simpler variant} {\color{blue} Maybe we should add here that our proofs of Proposition~\ref{prop:UnboundedStab} and Lemma~\ref{lm:HoweMoore} are elementary, hence we include them?}
%\textcolor{orange}{Yes, I can do this--we will probably also want to describe our motivation for tree products in the introduction}

\subsection{Convergence of diagrams} \label{sec:ConvergenceDiagrams}

Recall the definition of the Voronoi diagram $\mathrm{Vor}(\mathbf{X}^{(\lambda)})=\big\{C_1^{(\lambda)},C_2^{(\lambda)},\ldots\big\}$ from~\eqref{def:VoronoiDiagram}, which gives a partition 
$$
G=\bigsqcup_{i\in \mathbb{N}} C^{(\lambda)}_i.
$$
We assign to $\mathbf{X}^{(\lambda)}$ a subset $P(\mathbf{X}^{(\lambda)}) \in \{0,1\}^{E(G)}$ of edges which encodes this partition by declaring
$$[v,w]\in P(\mathbf{X}^{(\lambda)}) \text{ if and only if } [v,w]\in E(G) \text{ and } C^{(\lambda)}(v)=C^{\lambda}(w),$$
where we recall that $C^{(\lambda)}(v)$ is the Voronoi cell that contains $v$.
Define $\mathcal{P}^{(\lambda)}$ to be the distribution on $\{0,1\}^{E(G)}$ induced by the assignment $\mathrm{Vor}(\mathbf{X}^{(\lambda)})\mapsto P(\mathbf{X}^{(\lambda)})$. 

As $\{0,1\}^{E(G)}$ is a compact metric space and by Theorem~\ref{thm:SubseqLimit}, every sequence $\lambda_n\to 0$ has a subsequence $\lambda_{n_k}\to 0$ such that $\mu_{\lambda_{n_k}}\xrightarrow{v} \mu$ and $\mathcal{P}^{(1-e^{-\lambda_{n_k}})}\xrightarrow{v} \mathcal{P}$ for some distribution $\mathcal{P}$ on $\{0,1\}^{E(G)}$.
Next, we show in Theorem~\ref{thm:Diagrams} that $\mathcal P$ coincides with the distribution obtained by encoding the ideal Poisson-Voronoi diagram that corresponds to $\mu$.
As the proof of Theorem~\ref{thm:Diagrams} is standard, we only include it in Appendix~\ref{app:Trees}.

Suppose that $\lambda_n\to 0$ is such that $\mu_{\lambda_{n_k}}\xrightarrow{v} \mu$.
Let ${\bf X}=\{(X_1,Y_1),(X_2,Y_2),\dots\}$, where $\{X_1,X_2,\dots\}$ is the Poisson point process on ${\bf D}$ with intensity $\mu$ that is ordered by $X_1(o)\le X_2(o)\le \dots$ and $\{Y_1,Y_2,\dots\}$ are iid ${\rm Unif}[0,1]$-labels.
By Theorem~\ref{thm:SubseqLimit}~(iii), we have that a.s.
$$c_v\coloneqq \min\{X_i(v):i\in \mathbb{N}\}\not=-\infty \text{ and } |\{i\in \mathbb{N}:X_i(v)=c_v\}|<\infty$$
for every $v\in G$.
We may thus define the \emph{ideal Voronoi diagram} 
$$
\mathrm{Vor}(\mathbf{X})=\big\{C_1,C_2,\ldots\big\},
$$
where $v\in C_i$ if $c_v=X_i(v)$ and $Y_i$ is minimal among $\{Y_j:c_v=X_j(v)\}$. Let $\mathcal{P}^\mu$ be the distribution on $\{0,1\}^{E(G)}$ induced by the map $\mathbf{X}\mapsto P(\mathbf{X})\in \{0,1\}^{E(G)}$ that encodes the partition $G=\bigsqcup_{i\in \mathbb{N}} C_i$ defined as above.
Similarly, we define $\mathcal{P}^{\mu_{\lambda_n}}$ for every $n\in \mathbb{N}$.

The proof of the following result can be found in Appendix~\ref{app:Trees}.

\begin{theorem}\label{thm:Diagrams}
    Let $\lambda_n\to 0$ be such that $\mu_{\lambda_n}\xrightarrow{v} \mu$ and $\mathcal{P}^{(1-e^{-\lambda_{n}})}\xrightarrow{v} \mathcal{P}$.
    Then $\mathcal{P}=\mathcal{P}^\mu$.        
\end{theorem}

We say that two disjoint sets $C,D\subseteq G$ are \emph{neighbors} if there are $v\in C$, $w\in D$ that form an edge $[v,w]$ in $G$.

\begin{proposition}\label{pr:CellsAreCells}
    Let $\lambda_n\to 0$ be such that $\mu_{\lambda_n}\xrightarrow{v} \mu$.
    Then a.s. $C_i$ is non-empty and connected for every $i\in \mathbb{N}$, and $C_i$ and $C_j$ are neighbors for every $i\not=j\in \mathbb{N}$. 
\end{proposition}
\begin{proof}
    By Theorem~\ref{thm:SubseqLimit}, we have that a.s. $X_j\in (\partial \mathbb{T}_d)^k\times \mathbb{Z}$ for every $j\in \mathbb{N}$ and $|F^i_v|<\infty$ for every $v\in G$ and $i\in \mathbb{N}$, where $F^i_v=\{j\in \mathbb{N}:X_j(v)\le X_i(v)\}$.
    
    Let $i\in \mathbb{N}$, we say that a path $(v_\ell)_\ell$ in $G$ \emph{approaches $X_i$} if $X_i(v_{\ell+1})=X_i(v_\ell)-1$.
    Note that this notion is much weaker than convergence.
    We claim that the following holds a.s. for every path $(v_\ell)_\ell$ that approaches $X_i$:
    \begin{itemize}
        \item [(a)] $\{v_\ell\}_\ell\subseteq \bigcup_{j\in F^i_{v_1}} C_j$,
        \item [(b)] there is $\ell_0\in \mathbb{N}$ such that $\{v_\ell\}_{\ell\ge \ell_0}\subseteq C_i$,
        \item [(c)] if $v_1\in C_i$, then $\{v_\ell\}_\ell\subseteq C_i$. 
    \end{itemize}
    To see (a), take any $m\in \mathbb{N}$ such that $X_i(v_1)<  X_m(v_1)$ and $\ell\in \mathbb{N}$.
    Then we have that $X_m(v_\ell)\ge X_m(v_1)-\ell+1$ as $X_m$ is $1$-Lipschitz.
    By the assumption, we have that $X_i(v_\ell)=X_i(v_1)-\ell+1$, which shows (a).
    Property (c) is proven along the same lines using additionally the fact that $Y_i$ is minimal among $\{Y_j:j\in F^i_{v_1}\}$.
    To see (b), first observe that by the definition of the graph product, we have that $v_\ell$ and $v_{\ell+1}$ differ in a single coordinate in $G$.
    Let $1\le k'\le k$ be such that the set $A_{k'}$ of those $\ell\in \mathbb{N}$ such that $v_{\ell}$ and $v_{\ell+1}$ differ in coordinate $k'$ is infinite.
    For $j\in F^i_{v_1}$ we write $\psi^j_{k'}$ for the element of $\partial \mathbb{T}_d$ that represents the $k'$th coordinate of $X_j$.
    By Theorem~\ref{thm:SubseqLimit}, we have that a.s. $\psi^j_{k'}\not=\psi^i_{k'}$ whenever $i\not= j$.
    This implies that $X_j(v_{\ell})\to \infty$ as $\ell\to\infty$ as $X_j(v_\ell)+1=X_{j}(v_{\ell+1})$ for every large enough $\ell\in A_{k'}$.
    Combined with (a), we see that (b) holds.
    
    Now the proof can be finished as follows.
    It follows from (b) that $C_i\not=\emptyset$ for every $i\in \mathbb{N}$.
    Given $v,w\in C_i$, there are paths $(v_\ell)_\ell$ and $(w_m)_m$ that start at $v$ and $w$ respectively and approach $X_i$ with the property that there are $\ell_0,m_0\in \mathbb{N}$ such that $v_{\ell_0+n}=w_{m_0+n}$ for every $n\in \mathbb{N}$.
    Indeed, let $r=(r_1,\dots,r_k)\in (\mathbb{T}_d)^k$ be such that $r_j$ separates $\{X_i\}$ from $\{v^j,w^j\}$ for every $1\le j\le k$, where $v=(v^1,\dots,v^k)$ and $w=(w^1,\dots, w^k)$.
    Note that such $r$ exists as we work in $(\mathbb{T}_d)^k$.
    Let $\ell_0=\dist(v,r)$, and define a path $(v_\ell)_\ell$ such that $\dist(v_p,r)>\dist(v_{p+1},r)$ for every $0\le p< \ell_0$ and $(v_\ell)_{\ell\ge \ell_0}$ is a fixed path that starts at $r$ and approaches $X_i$.
    By the definition of $r$, we have that $(v_\ell)_\ell$ approaches $X_i$ as well.
    Analogous construction for $w$ gives the claim.
    It follows from (c) that $v,w$ are connected with a path within $C_i$.

    For the second claim of the proposition, use Theorem~\ref{thm:UnboundedTouching} to find $v\in W^{\bf X}_2(X_i,X_j)$.
    This implies that $F^i_v=\{i,j\}$.
    Without loss of generality we may assume that $v\in C_j$.
    Let $(v_\ell)_\ell$ be a path in $G$ that starts at $v$ and approches $X_i$.
    It follows from (a) and (b) that there is $\ell_0>1$ such that $v_{\ell_0}\in C_i$ and $\{v=v_1,\dots,v_{\ell_0}\}\subseteq C_i\cup C_j$.
\end{proof}

Theorem~\ref{thm:Diagrams} allows to formulate a small intensity analogue of Proposition~\ref{pr:CellsAreCells}.

\begin{theorem}\label{thm:LocalTouching}
    Let $\epsilon,R>0$.
    Then there is $\lambda_0>0$ such that
    \begin{align*} 
    & \inf_{0<\lambda\le \lambda_0} \mathbb{P}\left(C^{(\lambda)}_i\cap B_{R}(o)\not=\emptyset\not=C^{(\lambda)}_j\cap B_{R}(o) \ \Rightarrow \ C^{(\lambda)}_i \text{ and } C^{(\lambda)}_j\text{ are neighbors}\right) \\
    & \qquad >1-\epsilon.
    \end{align*}
\end{theorem}
\begin{proof}
    Suppose for a contradiction that there is $\epsilon,R>0$ and a sequence $\lambda_n\to 0$ such that
    $$\mathbb{P}\left(C^{(\lambda_n)}_i\cap B_{R}(o)\not=\emptyset\not=C^{(\lambda_n)}_j\cap B_{R}(o) \ \Rightarrow \ C^{(\lambda_n)}_i \text{ and } C^{(\lambda_n)}_j\text{ are neighbors}\right)\le 1-\epsilon$$
    holds for every $n\in \mathbb{N}$.
    Define $\nu_n=-\log(1-\lambda_n)$ and note that $\nu_n\to 0$.
    After possibly passing to a subsequence, we may assume that
    $$\mu_{\nu_n}\xrightarrow{v} \mu \text{ and } \mathcal{P}^{(\lambda_n)}=\mathcal{P}^{(1-e^{-\nu_n})}\xrightarrow{v} \mathcal{P}.$$
    Using Theorem~\ref{thm:Diagrams}, we know that $\mathcal{P}=\mathcal{P}^\mu$.
    For $S>0$, we define the event $A(S)\subseteq \{0,1\}^{E(G)}$, where $\mathcal{E}\in A(S)$ if for all $v,w\in B_{R}(o)$ there is a path in $B_{R+S}(o)$ in $G$ connecting $v$ and $w$ such that at most one edge $e$ on this path satisfies $\mathcal{E}(e)=0$.
    By Proposition~\ref{pr:CellsAreCells}, we find $S>0$ such that
    $$\mathcal{P}(A(S))>1-\epsilon/2.$$
    As $\mathcal{P}^{(\lambda_{n})}\xrightarrow{v} \mathcal{P}$ and $A(S)$ is clopen, we must have that
    $$\mathcal{P}^{(\lambda_n)}(A(S))>1-\epsilon$$
    for some $n\in \mathbb{N}$.
    This is a contradiction as conditioned on $A(S)$ all cells that intersect $B_R(o)$ are neighbors.
\end{proof}

We finish this subsection by the following intuitively clear result.

\begin{lemma}\label{lem:NumberOfCells}
    For every $\epsilon,K>0$ there is $R>0$ and $\lambda_0>0$ such that 
    $$\inf_{0<\lambda\le \lambda_0}\mathbb{P}(\#\text{ cells } C^{(\lambda)}_i \text{ intersecting } B_R(o)>K )>1-\epsilon$$
\end{lemma}
\begin{proof}
    Suppose for a contradiction that there is $\epsilon,K>0$ that do not satisfy the claim.
    In particular, for any sequence $R_n\to \infty$ there is a sequence $\lambda_n\to 0$ such that 
    $$\mathbb{P}(\#\text{ cells } C^{(\lambda_n)}_i \text{ intersecting } B_{R_n}(o)>K )\le 1-\epsilon.$$
    As in the proof of Theorem~\ref{thm:LocalTouching}, we may assume that 
    $$\mu_{\nu_n}\xrightarrow{v} \mu \text{ and } \mathcal{P}^{(\lambda_n)}\xrightarrow{v} \mathcal{P}^{\mu},$$
    where $\nu_n=-\log(1-\lambda_n)$.

    For $R>0$, define the clopen event $C(R)\subseteq \{0,1\}^{E(G)}$ as follows.
    Let $\mathcal{E}\in C(R)$ if there is $\{v_1,\dots,v_{K+1}\}\subseteq B_R(o)$ such that for every $1\le i\not=j\le K+1$ there is a path $P_{i,j}\subseteq B_{R}(o)$ connecting $v_i$ and $v_j$ such that $\mathcal{E}$ is missing exactly one edge of $P_{i,j}$.
    Note that if $\mathcal{E}\in C(R)$ encodes a partition $G=\bigsqcup_{i\in \mathbb{N}} D_i$ into infinitely many connected sets, then at least $K+1$ distinct parts intersect $B_R(o)$.
    Using this observation the proof is finished as follows.
    
    By Proposition~\ref{pr:CellsAreCells}, there is $R>0$ such that
    $$\mathcal{P}^\mu(C(R))>1-\epsilon/2.$$
    As $\mathcal{P}^{(\lambda_{n})}\xrightarrow{v} \mathcal{P}^\mu$ and $C(R)$ is clopen, it follows that 
    $$\mathcal{P}^{(\lambda_n)}(C(R))>1-\epsilon$$
    for large enough $n\in \mathbb{N}$, and that is a contradiction.
\end{proof}

\subsection{Proof of Theorem~\ref{thm:localuniqueness}}

\begin{proof}[Proof of Theorem~\ref{thm:localuniqueness}]
    Given $p_0>0$ and $\varepsilon>0$, choose $K\in\mathbb N$ such that 
    $$
    1-(1-p_0)^K(1-\varepsilon/2)>1-\varepsilon.
    $$ 
    By Lemma~\ref{lem:NumberOfCells}, there exists $R>0$ and $\lambda_1>0$ such that
       $$\inf_{0<\lambda\le \lambda_1}\mathbb{P}(\#\text{ cells } C^{(\lambda)}_i \text{ intersecting } B_R(o)>K )>1-\varepsilon/4.$$
    By Theorem~\ref{thm:LocalTouching}, there exists $\lambda_0>0$, and without loss of generality $\lambda_0\le\lambda_1$, such that
     \begin{align*} 
    & \inf_{0<\lambda\le \lambda_0} \mathbb{P}\left(C^{(\lambda)}_i\cap B_{3R}(o)\not=\emptyset\not=C^{(\lambda)}_j\cap B_{3R}(o) \ \Rightarrow \ C^{(\lambda)}_i \text{ and } C^{(\lambda)}_j\text{ are neighbors}\right) \\
    & \qquad >1-\varepsilon/4.
    \end{align*}
    Hence 
    \begin{align*} 
    & \inf_{0<\lambda\le \lambda_0} \mathbb{P}\bigg(C^{(\lambda)}_i\cap B_{3R}(o)\not=\emptyset\not=C^{(\lambda)}_j\cap B_{3R}(o) \ \Rightarrow \ C^{(\lambda)}_i \text{ and } C^{(\lambda)}_j\text{ are neighbors,} \\
    & \qquad \qquad \#\text{ cells } C^{(\lambda)}_i \text{ intersecting } B_R(o)>K \bigg) >1-\varepsilon/2.
    \end{align*}
    Conditional on the event in the above display, independently coloring each cell black with probability $p_0$ yields at most one black cluster in $B_{3R}(o)$, which has non-empty intersection with $B_R(o)$ with probability at least $1-(1-p_0)^K$. The claim now follows from the choice of $K$.
\end{proof}

\begin{remark}\label{rem:GeneralProducts} 
As alluded to in Remark~\ref{rem:Generalization}, Theorem~\ref{thm:UnboundedTouching} can be extended to graphs of the form $G\coloneqq (\mathbb T_d)^k\times H$, where $H$ is a locally finite, connected graph such that the volume of the ball $B_r$ of radius $r$ in $H$ satisfies $|B_r|\in O((d-1)^r)$. More precisely, the proof of Theorem~\ref{thm:SubseqLimit} shows that subsequential limits are supported on $U\coloneqq (\partial \mathbb T_{d})^k\times \overline H\times \mathbb Z$ and decompose as product measures with harmonic measure in the first $k$ coordinates and a measure $\theta$ on $\overline{H}\times\mathbb Z$ satisfying 
$$
\forall m: c(d-1)^m\leq\theta(\overline{H}\times\{m\}) \leq C (d-1)^m
$$
for some $0<c,C<\infty$. The proof of Proposition~\ref{prop:UnboundedStab} shows that if $(\phi_1,\ldots,\phi_k,\psi,m)$ and $(\phi_1',\ldots,\phi_k',\psi',m')$ are elements of $U$ such that $\phi_i\ne\phi_i'$ for at least two $i\le k$, then the $\Gamma$-stabilizer, where $\Gamma:=({\rm Aut(\mathbb T_d)})^k\times\{{\rm id}_H\}$, of these two points is unbounded. The statement of Lemma~\ref{lm:HoweMoore} also remains valid when we replace ${\rm dist}(u,v)$ for $u,v\in G$ with the distance between their projections to the first $k$ coordinates. Thus the proof can be finished in the same way as for $(\mathbb T_d)^k$. Finally, it is not difficult to check that the results in Section~\ref{sec:ConvergenceDiagrams}, and hence the proof of Theorem~\ref{thm:localuniqueness}, extend to this setting as well. 
\end{remark}

\section{Proofs of the main results}

\begin{proof}[Proof of Theorem~\ref{thm:mainTrees}] Let $d\ge 3$, $k\ge 2$ and $G\coloneqq \mathbb T_d\times\ldots\times \mathbb T_d$ be the $k$-fold graph product. Theorem~\ref{thm:localuniqueness} shows that $G$ satisfies the assumptions of Theorem~\ref{thm:LocalToGlobal}. Theorem~\ref{thm:LocalToGlobal} shows that for every $p_0\in(0,1]$, there exists $\lambda_0>0$ such that  
$$
\inf_{u,v\in V(G)} \mathbb P_{p_0}^{(\lambda)}( u\leftrightarrow v) >0
$$
for all $\lambda<\lambda_0$. By Theorem~\ref{thm:UniquenessLRO}, we obtain that $p_u(\lambda)\to0$ as $\lambda\to0$.
\end{proof}

\begin{proof}[Proof of Theorem~\ref{thm:mainSymm}] The assumptions imply that $M$ is the symmetric space of a connected higher rank semisimple real Lie group. In this setting, it was shown in~\cite[Section 3 \& 4]{GR25} that Poisson--Voronoi percolation on $M$ satisfies local uniqueness at low intensities in the sense of \eqref{equ:localuniquenessSymmSpace}. Theorem~\ref{thm:LocalToGlobalSymm} shows that for every $p_0\in(0,1]$, there exists $\lambda_0>0$ such that 
$$
\inf_{x,y\in M} \mathbb P_{p_0}^{(\lambda)} ( x\leftrightarrow y) >0
$$
for all $\lambda<\lambda_0$. The claim follows from~\cite[Theorem 1.9]{GR25}.
\end{proof}

\begin{proof}[Proof of Theorem~\ref{thm:mainSimplyConn}] The assumptions guarantee that $M$ splits as a Riemannian product of symmetric spaces associated to connected simple real Lie groups whose sum of ranks is greater than or equal to $2$. If there are at least two factors with rank greater than or equal to $1$, the assertion follows from Theorem~\ref{thm:mainSymm} as we can split into a product of $X$ and $Y$, which are symmetric spaces of non-compact connected semisimple real Lie groups $G$ and $H$. If there is at most one factor with rank greater than or equal to $1$, this factor has rank greater than or equal to $2$, and thus property~(T). Since compact groups have property~(T), it follows that $M$ is the symmetric space of a connected higher rank semisimple real Lie group with property~(T).
In this case, the claim follows from \cite[Theorem~1.1]{GR25}.
\end{proof}

\section{Future perspectives}

\begin{question}[\textsc{Strong product of $\mathbb{T}_d\times \mathbb{Z}$}] \label{q:strongproduct}
    Consider the strong product of $\mathbb T_d$ and $\mathbb Z$, which is the graph structure on $\mathbb{T}_d\times \mathbb{Z}$ where each pair of distinct vertices $(v_1,w_1)$ and $(v_2,w_2)$ forms an edge if $\dist(v_1,v_2)\le 1$ and $\dist(w_1,w_2)\le 1$. Note that this graph corresponds to the $L^\infty$-product of the two metric spaces. 
    In this case, it can be proven that the IPVT only depends on the first coordinate a.s.
    That is, each cell is of the form $C\times\mathbb{Z}$, where $C\subseteq \mathbb{T}_d$.
    It follows that any two cells that touch have an unbounded border, but, of course, not every pair of cells touch. Additionally, Bernoulli site percolation on the Delaunay graph of the IPVT has $p_u$ equal to $1$.
    What is the value $\liminf_{\lambda\to 0} p_u(\lambda)$ in this model?
\end{question}

\begin{question}[\textsc{Sufficient conditions for a.s.~unbounded boundaries}]
In the light of the discussion in Section~\ref{sec:introdiscrete}, devise a sufficient criterion, which does not rely on specific properties of the underlying isometry group as in \cite{FMW23}, for having a.s.~unbounded borders between every pair of cells, resp.~every pair of cells with non-trivial intersection, in a generic metric space endowed with an infinite Radon measure (supporting a non-trivial IPVT). 
\end{question}

\begin{question}[\textsc{Discontinuous vanishing uniqueness thresholds}] Does there exist a non-amenable Cayley graph such that 
$$
\lim_{\lambda\to0} p_u(\lambda) = 0 < p_u(0),
$$
where $p_u(0)$ is the uniqueness threshold of Bernoulli site percolation on the Delaunay graph of the IPVT?
\end{question}

\begin{question}[\textsc{Other generating sets}]
    The unbounded borders phenomenon is sensitive to the local geometry. Indeed, contrast the graph product of $\mathbb T_3$~and~$\mathbb T_3$ with the strong product, cf.~Question~\ref{q:strongproduct}. These are Cayley graphs of the same group and, in particular, quasi-isometric. While the former satisfies the unbounded borders phenomenon as shown in Theorem~\ref{thm:UnboundedTouching}, the latter does not as can be seen by studying sequences of $L^\infty$-horofunctions centered around points converging to infinity in suitable directions. This observation inspires the following question: How strong is the dependence of the phenomenon ``$p_u(\lambda)\to0$'' on the choice of generating set?
\end{question}

\bibliographystyle{alpha}
\bibliography{percbiblio}

\newpage

\appendix

\section{Additional proofs from Section~\ref{sec:Basic}}\label{app:Proofs}
In the setting of Section~\ref{sec:model}, denote by $\mathcal G^{(\lambda)}=(V^{(\lambda)},E^{(\lambda)})$ the graph with vertices
$$
V^{(\lambda)}\coloneqq  \mathrm{Vor}(\mathbf{X}^{(\lambda)}\cup\{o\})
$$ 
and edges between every pair of cells sharing a neighbor in the original graph. This graph is called the {\bf Delaunay graph} associated to the tessellation $\mathrm{Vor}(\mathbf{X}^{(\lambda)}\cup\{o\})$. Here we work with
$$
\mathbf{X}^{(\lambda)}_0\coloneqq \mathbf{X}^{(\lambda)}\cup\{o\}
$$
instead of $\mathbf{X}^{(\lambda)}$ to treat the Delaunay graph in the framework of unimodular random rooted graphs. We assume that $o$ is equipped with an independent ${\rm Unif}[0,1]$-label. Note that $\mathbf X_0^{(\lambda)}$ coincides with $\mathbf X^{(\lambda)}$ conditioned to have a point at the origin.

Note that Bernoulli--Voronoi percolation defined w.r.t.~$\mathbf{X}^{(\lambda)}_0$ and $p$-Bernoulli site percolation $\mathcal G^{(\lambda)}[p]$ on the random graph $\mathcal G^{(\lambda)}$ are closely related. Indeed, a vertex $v\in V$ belongs to an infinite Bernoulli--Voronoi percolation cluster if and only if its Voronoi cell $C^{(\lambda)}(v)$ belongs to an infinite $\mathcal G^{(\lambda)}[p]$-cluster. As a consequence, the number of infinite clusters and the phase transition in the parameter $p$ are the same. 

\begin{lemma} \label{lm:Palm}
Let $G=(V,E,o)$ be a locally finite, connected, infinite, transitive, unimodular rooted graph. Let $A$ be an ${\rm Aut}(G)$-invariant event and let $A_0$ be the restriction of $A$ to configurations containing the origin. Then 
$$
\mathbb P(\mathbf{X}^{(\lambda)}_0\in A_0)=\mathbb P(\mathbf{X}^{(\lambda)}\in A)\in\{0,1\}.
$$
\end{lemma}
\begin{proof} Since $\mathbf X^{(\lambda)}$ can be obtained from the Poisson point process with iid  ${\rm Unif}[0,1]$-labels by forgetting multiplicities, this follows from similar arguments as in~\cite[Lemma~6.6]{GR25}.
\end{proof}

\begin{proposition} \label{prop:BernoulliDelaunay} 
Let $G=(V,E,o)$ be a locally finite, connected, infinite, transitive, unimodular rooted graph. Fix $\lambda>0$. Then the following hold:
\begin{enumerate}
\item[{\rm (i)}] the random rooted graph $(\mathcal G^{(\lambda)},C^{(\lambda)}(o))$ is unimodular and extremal;
\item[{\rm (ii)}] the number of infinite clusters in $\mathcal G^{(\lambda)}[p]$ is $0,1$ or $\infty$ a.s.;
\item[{\rm (iii)}] there exists a constant $p_c\in[0,1]$ such that $\mathcal G^{(\lambda)}[p]$ has infinite clusters a.s.~if $p>p_c$ and does not have infinite cluster a.s.~if $p<p_c$;
\item[{\rm (iv)}] there exists a constant $p_u\in[0,1]$ such that $\mathcal G^{(\lambda)}[p]$ has a unique infinite cluster a.s.~if $p>p_u$ and does not have a unique infinite cluster a.s.~if $p<p_u$.
\end{enumerate}
\end{proposition}

\begin{proof}    
(i). Note that $\mathcal G^{(\lambda)}$ is supported on locally finite, connected, infinite graphs by Proposition~\ref{prop:VoronoiBasics}~(i). 

We now show that the random rooted graph $(\mathcal G^{(\lambda)},C^{(\lambda)}(o))$ is unimodular by verifying that it obeys the Mass Transport Principle, cf.~\cite{AL07}. Let $f: \mathcal G_{\bullet\bullet} \to[0,\infty]$ be a Borel measurable function on the space of (isomorphism classes of) bi-rooted graphs equipped with the local metric. Define 
\begin{equation*}
m(u,v,\mathbf X^{(\lambda)}) \coloneqq  \mathbf{1}\big\{ u,v \in \mathbf X^{(\lambda)} \big\} f\big( \mathcal G^{(\lambda)}, C^{(\lambda)}(u), C^{(\lambda)}(v) \big).
\end{equation*}
Since $f(G,o,x)$ depends only on the isomorphism class, it follows that 
$$
m(gu,gv,g\mathbf X^{(\lambda)})=m(u,v,\mathbf X^{(\lambda)}) 
$$
for every $g\in \mathrm{Aut}(G)$. In other words, $m$ is $\mathrm{Aut}(G)$-diagonally invariant. From the definition of $m$ and the Mass Transport Principle for transitive unimodular graphs, we thus conclude that
\begin{align*}
\mathbb E \bigg[ \sum_{ C\in V(\mathcal G^{(\lambda)})} f\big(\mathcal G^{(\lambda)},C^{(\lambda)}(o), C\big) \bigg] & = \mathbb E \bigg[ \sum_{v\in V(G)} m\big(o,v, \mathbf X^{(\lambda)}\big) \Big| o\in \mathbf X^{(\lambda)} \bigg] \\
&= \frac{1}{\lambda} \mathbb E \bigg[ \sum_{v\in V(G)} m\big(o,v, \mathbf X^{(\lambda)}\big) \bigg] \\
& = \frac{1}{\lambda} \mathbb E \bigg[ \sum_{v\in V(G)} m\big(v,o, \mathbf X^{(\lambda)}\big) \bigg] \\ 
& = \mathbb E \bigg[ \sum_{x\in V(\mathcal G^{(\lambda)})} f\big(\mathcal G^{(\lambda)},C, C^{(\lambda)}(o)\big) \bigg].
\end{align*}
This proves unimodularity.

Let $\mathcal I$ denote the $\sigma$-field of events in $\mathcal G_{\bullet}$ (the space of (isomorphism classes of) rooted graphs equipped with the local metric) which are invariant under non-rooted isomorphisms. By~\cite[Theorem 4.7]{AL07}, it suffices to show that $\mathcal I$ is trivial. This follows from ergodicity of $\mathbf X^{(\lambda)}$. Note that this also implies (iii).

Finally, (ii) and (iv) follow from \cite[Corollary~6.9 \& Theorem~6.7]{AL07}.
\end{proof}

\begin{proof}[Proof of Proposition~\ref{prop:Phases}] Combine Proposition~\ref{prop:BernoulliDelaunay} and Lemma~\ref{lm:Palm}.
\end{proof}

\begin{proof}[Proof of Theorem \ref{thm:UniquenessLRO}]

Let 
\begin{equation*}
p_{\rm  LRO}(\lambda) \coloneqq  \inf \Big\{ p \, : \, \inf_{u,v\in V} \mathbb P_{p}^{(\lambda)}( u \leftrightarrow v)>0 \Big\}
\end{equation*}
denote the infimum in \eqref{equ-UniquenessLRO}. 

The fact that $p_u(\lambda) \geq p_{\rm  LRO}(\lambda)$ follows by a standard argument: If $p>p_u(\lambda)$, Proposition~\ref{prop:BernoulliDelaunay}~(iii) implies that $\omega_p^{(\lambda)}$ has a unique infinite cluster with positive probability. By transitivity, 
$\mathbb P_p^{(\lambda)}(|C(v)|=\infty)=\mathbb P_p^{(\lambda)}(|C(o)|=\infty)>0$ for all $v\in V$. The FKG-inequality (Proposition~\ref{prop:VoronoiBasics}~(ii)) implies that
$$
\inf_{u,v\in V} \mathbb P_{p}^{(\lambda)}( u \leftrightarrow v) \ge \inf_{u,v\in V} \mathbb P_{p}^{(\lambda)}(|C(u)|, |C(v)| =\infty) \ge \mathbb P_p^{(\lambda)}(|C(o)|=\infty)^2>0.
$$
Hence $p>p_{\rm LRO}(\lambda)$. 

Conversely, if $p>p_{\rm LRO}(\lambda)$, then (by monotonicity in $p$)
$$
\inf_{u,v\in V} \mathbb P_{p}^{(\lambda)}( u \leftrightarrow v)>0.
$$
We will now show that $\omega_{q}^{(\lambda)}$ has a unique infinite cluster for every $q>p$. Clearly, this implies that $p\ge p_u(\lambda)$ and will thus prove the converse inequality.

Fix $q>p$. By monotonicity, 
$$
\inf_{u,v\in V} \mathbb P_{q}^{(\lambda)}( u \leftrightarrow v)>0,
$$
which implies that $\omega_{q}^{(\lambda)}$ has some infinite cluster with positive probability. 

By Proposition~\ref{prop:BernoulliDelaunay}~(ii), it thus suffices to show that $\omega_q^{(\lambda)}$ has finitely many infinite clusters. Recall the definition of {\bf cluster frequencies} from \cite[Lemma 4.2]{LS99}.

\begin{lemma}[\textsc{Cluster frequencies}] \label{lm:frequ}
There exists a ${\rm Aut}(G)$-invariant measurable function ${\rm freq}: \{0,1\}^V \to[0,1]$ with the following property. Suppose that $\mu$ is the law of an ${\rm Aut}(G)$-invariant site percolation $\omega$ on~$G$. Let $v\in V$ and let $\mathbb P_v$ be the law of simple random walk $(X_n)_{n=1}^\infty$ on $G$ started in $v$. Then, $\mu\otimes\mathbb P_v$-almost surely, 
$$
\lim_{n\to\infty} \frac{1}{n} \sum_{i=0}^{n-1} \mathbf 1_{\{ X_i\in C\}} = {\rm frequ}(C) \quad \text{for every cluster~} C \text{ of } \omega.
$$ 
\end{lemma}

With this preparation, the fact that $\omega_q^{(\lambda)}$ has only finitely many clusters may be proven similarly as in \cite{GR25}. We thus include only a sketch of the argument: 

Long-range order for $\omega_p^{(\lambda)}$ implies that the expected frequency of the $\omega_p^{(\lambda)}$-cluster of the origin is positive. In particular, there is a cluster of positive frequency and thus finitely many clusters of maximal frequency almost surely. Call the clusters maximizing frequency special. The same argument as in \cite[pp. 38-39]{GR25} shows that every infinite $\omega_q^{(\lambda)}$-cluster must contain a special $\omega_p^{(\lambda)}$-cluster.\footnote{This argument is based on the proof of uniqueness monotonicity \cite{HP99} given in \cite[Theorem~5.4]{HJ06}.} In particular, $\omega_q^{(\lambda)}$ has finitely many infinite clusters, which concludes the proof.
\end{proof}

\section{Proofs from Section~\ref{sec:Trees}}\label{app:Trees}

\begin{proof}[Proof of Proposition~\ref{pr:BasicCorona}]
    (i).
    Let $a_n=(a_{n,i})_{i=1}^k\in G$ and $(\ell_n)_{n\in \mathbb{Z}}$ be such that
    $$\dist(v,a_n)+\ell_n\to f(v)$$
    for every $v\in G$, and set $m=f({\bf o})$.
    After possibly passing to a subsequence, we may assume that for every $i\in \{1,\dots,k\}$ there is $\psi_i\in \overline{\mathbb{T}_d}$ such that either $a_{n,i}=\psi_i\in \mathbb{T}_d$ for every $n\in \mathbb{N}$ or
    $$a_{n,i}\to\psi_i\in \partial\mathbb{T}_d.$$
    We claim that $\psi=(\psi_1,\dots,\psi_k)$ and $m$ works as required.

    Fix $v=(v_1,\dots,v_k)\in G$ and take $n\in \mathbb{N}$ large, so that for every $i\in \{1,\dots,k\}$ such that $\psi_i\in \partial\mathbb{T}_d$, we have that $a_{n,i}\in \mathbb{T}_d$ separates $\{o,v_i\}$ and $\psi_i$.
    Then 
    \begin{equation}
        \begin{split}
            f(v)= & \ \sum_{i=1}^k \dist(v_i,a_{n,i})+\ell_n\\
            = & \ \left(\sum_{i=1}^k \dist(v_i,a_{n,i})-\dist(o,a_{n,i})\right)+\sum_{i=1}^k \dist(o,a_{n,i}) +\ell_n\\
            = & \ h^\psi(v)+f({\bf o})=h^\psi(v)+m,
        \end{split}
    \end{equation}
    where the first equality and third equality hold for large enough $n\in \mathbb{N}$.

    (ii).
    For every $1\le i\le k$, define a sequence $(a_{n,i})_n\subseteq (\mathbb{T}_d)^k$ such that $a_{n,i}=\psi_i$ for every $n\in\mathbb{N}$ whenever $\psi_i\in \mathbb{T}_d$ and $a_{n,i}\to \psi_i$ otherwise.
    Set $a_n=(a_{n,1},\dots,a_{n,k})$.
    We claim that $d(v,a_n)+\ell_n\to h^\psi+m$ for every $v=(v_1,\dots,v_k)\in (\mathbb{T}_d)^k$, where $d({\bf o},a_n)+\ell_n=m$.
    Indeed, we have
    \begin{equation*}
        %\begin{split}
            d(v,a_n)+\ell_n=  m+\sum_{i=1}^k(d(v_i,a_{n,i})-d(o,a_{n,i}))
            \to  h^\psi(v)+m
        %\end{split}
    \end{equation*}
    by the definition of $h^\psi$.

    (iii).
    Let $(\psi,m),(\phi,n)\in (\partial \mathbb{T}_d)^k\times \mathbb{Z}$ and suppose that $h^\psi+m=h^\phi+n$.
    Then $m=n$ as $m=h^\psi({\bf o})+m=h^\phi({\bf o})+n=n$.
    Let $w_n=(a_n,o,\dots,o)\in (\mathbb{T}_d)^k$ be such that $a_n\to \psi_1$, where $\psi=(\psi_1,\dots,\psi_k)$.
    Then we have that
    $$h^\phi(w_n)+n=h^\psi(w_n)+m=-d(o,a_n)+m\to -\infty$$
    by the definition of $h^\psi$.
    Note that if $\phi_1\not=\psi_1$, where $\phi=(\phi_1,\dots,\phi_k)$, there exists $r\in \mathbb{T}_d$ separating $\{b_n\}_n\cup\{{\bf o}\}$ from $\psi$, where $b_n\to \phi_1$.
    Then we get that
    $$h^\phi(w_n)+n=d(b_n,r)-d(o,r)+n=+\infty.$$
    Hence $\phi_1=\psi_1$ and the first claim follows.

    Now let $(\psi_n,m_n),(\psi,m)\in (\partial \mathbb{T}_d)^k\times \mathbb{Z}$ be such that $(\psi_n,m_n)\to (\psi,m)$.
    We can assume without loss of generality that $m_n=m$ for every $n\in \mathbb{N}$.
    Let $v=(v_1,\dots,v_k)\in (\mathbb{T}_d)^k$ and fix $b=(b_1,\dots,b_k)\in (\mathbb{T}_d)^k$ such that, for $1\le i\le k$, $b_i$ separates $\{v_i,o\}$ from $\{\psi_{n,i}\}_{n\ge n_0}\cup\{\psi_{i}\}$ 
    for some $n_0\in \mathbb{N}$, where $\psi_n=(\psi_{n,1},\dots, \psi_{n,k})$ and $\psi=(\psi_{1},\dots, \psi_{k})$. 
    Then we have that 
    $$h^{\psi_n}(v)+m=\sum_{i=1}^k (d(v_i,b_i)-d(o,b_i))+m=h^{\psi}(v)+m$$
    for every $n\ge n_0$.
    This shows continuity.
    The fact that the map has closed range can be proved in a similar way.

    Finally, let $(\psi,m)\in (\partial \mathbb{T}_d)^k\times \mathbb{Z}$ and $\tau=(\tau_1,\dots,\tau_k)\in H$.
    Fix $v=(v_1,\dots,v_k)\in (\mathbb{T}_d)^k$.
    There exists $b=(b_1,\dots,b_k)\in (\mathbb{T}_d)^k$ such that $b_i$ separates $\{v_i,o,\tau_i(o)\}$ from $\tau_i(\psi_{i})$ for every $1\le i\le k$, where $\psi=(\psi_{1},\dots, \psi_{k})$.
    Then we have that
    \begin{equation*}
        \begin{split}
            h^{\tau(\psi)}(v) +h^\psi(\tau^{-1}({\bf o}))+m 
            = & \ \sum_{i=1}^k(d(v_i,b_i)-d(o,b_i)) \\
            & \ \ \ \ \ +\sum_{i=1}^k(d(\tau^{-1}_i(o),\tau^{-1}_i(b_i))-d(o,\tau^{-1}_i(b_i)))+m\\
%            = & \ \sum_{i=1}^k(d(\tau^{-1}_i(v_i),\tau^{-1}_i(b_i))-d(\tau^{-1}_i(o),\tau^{-1}_i(b_i))) +\sum_{i=1}^k(d(\tau^{-1}_i(o),\tau^{-1}_i(b_i))-d(o,\tau^{-1}_i(b_i)))+m\\
            = & \ \sum_{i=1}^k(d(\tau^{-1}_i(v_i),\tau^{-1}_i(b_i))-d(o,\tau^{-1}_i(b_i)))+m\\
            = & \ h^\psi(\tau^{-1}(v))+m.
        \end{split}
    \end{equation*}
    as $\tau^{-1}_i(b_i)$ separates $\{o,\tau^{-1}_i(o)\}$ from $\psi_{i}$ for every $1\le i\le k$.
\end{proof}

\begin{proof}[Proof of Theorem~\ref{thm:Diagrams}]
    In order to prove that $\mathcal{P}=\mathcal{P}^\mu$, it is enough to show that for every $\epsilon,R>0$ and a bond percolation configuration $\mathcal{E}$ in $G\upharpoonright B_R(o)$ we have that
    $$\lim_{n\to\infty}\mathcal{P}^{(1-e^{-\lambda_n})}(\widehat{\mathcal{E}})\ge \max\{0, \mathcal{P}^{\mu}(\widehat{\mathcal{E}})-\epsilon\},$$
    where $\widehat{\mathcal{E}}$ is the cylinder event that consists of percolation configurations that agree with $\mathcal{E}$ on $B_R(o)$.
    As a first step we show that it is enough to work with $\mathcal{P}^{\mu_{\lambda_n}}$.

    \begin{claim}\label{cl:LastClaim}
        Let ${\bf Z}^{(\lambda_n)}=\{(Z^{(\lambda_n)}_1,Y^{(\lambda_n)}_1),(Z^{(\lambda_n)}_2,Y^{(\lambda_n)}_2),\dots\}$ be the Poisson point process on ${\bf D}$ with intensity $\mu_{\lambda_n}$ and $\{Y_1,Y_2,\dots\}$ consists of iid ${\rm Unif}[0,1]$-labels.
        Then
        $$\mathcal{P}^{\mu_{\lambda_n}}=\mathcal{P}^{(1-e^{-\lambda_n})}$$
        for every $n\in \mathbb{N}$.
        In particular, $\mathcal{P}^{\mu_{\lambda_n}}\to \mathcal{P}.$
    \end{claim}
    \begin{proof}
        Observe that by the definition of $\iota_{t_{\lambda_n}}$, $\mathcal{P}^{\mu_{\lambda_n}}$ can be equivalently defined by taking the Voronoi diagram of the marked Poisson point process (with iid ${\rm Unif}[0,1]$-labels) ${\bf Z}^n=\{(Z^n_1,Y^n_1),(Z^n_2,Y^n_2),\dots\}$ on $G$ with intensity $\lambda_n\cdot \nu$, where $\nu$ is counting measure on $G$.
        
        Note that in the Voronoi diagram of ${\bf Z}^{n}$, if $Z^{n}_i=Z^{n}_j$, then the corresponding Voronoi cell of the point with the larger mark is empty.
        Let ${\bf X}^{n}$ be the marked point process obtained by refining the multiple points of ${\bf Z}^{n}$ by keeping only the point with the smallest mark.
        It is easy to see that ${\bf X}^{n}$ is a marked Bernoulli process with intensity $1-e^{-\lambda_n}$ and marks which are iid with a continuous distribution on $[0,1]$.
        This implies that $\mathcal P^{\mu_{\lambda_n}}=\mathcal P^{(1-e^{-\lambda_n})}$, and the assumption gives $\mathcal P^{\mu_{\lambda_n}}\to\mathcal P$.      
    \end{proof}

    Let $\epsilon,R>0$ and $\mathcal{E}$ be a percolation configuration in $G\upharpoonright B_R(o)$.
    By Claim~\ref{cl:LastClaim}, we are left with showing that
    $$\lim_{n\to\infty}\mathcal{P}^{\mu_{\lambda_n}}(\widehat{\mathcal{E}})\ge \max\{0, \mathcal{P}^{\mu}(\widehat{\mathcal{E}})-\epsilon\}.$$
    For $m,k\in \mathbb{Z}$ and $\ell\in \mathbb{N}$, consider a sequence $(\sigma,k,F_1,\dots,F_\ell)$, where
    \begin{enumerate}
        \item $\sigma$ is a permutation of $\{1,\ldots,\ell\}$,
        \item $F_j:B_R(o)\to \mathbb{Z}$ for every $1\le j\le \ell$ with the property that $m=F_1(o)\le F_2(o)\le \dots \le F_\ell(o)$ and $F_{\ell}(o)\le m+2R$,
        \item $k> m+2R$,
        \item $(\sigma,F_1,\dots,F_\ell)$ determines $\mathcal{E}$.
    \end{enumerate}
    Define $\mathcal{A}(\sigma,k,F_1,\dots,F_\ell)$ to be the event that $X_i\upharpoonright B_R(o)=F_i$ for $1\le i\le \ell$, $X_{\ell+1}(o)=k$ and the ordering of $(Y_i)_{i=1}^\ell$ induces the permutation $\sigma$.
    Similarly, define $\mathcal{A}^n(\sigma,k,F_1,\dots,F_\ell)$ for the marked Poisson point process with intensity $\mu_{\lambda_n}$.
    By $\mu_{\lambda_n}\xrightarrow{v} \mu$ we have that
    $$\mathbb{P}(\mathcal{A}^n(\sigma,k,F_1,\dots,F_\ell))\to \mathbb{P}(\mathcal{A}(\sigma,k,F_1,\dots,F_\ell))$$
    as agreeing with $(F_1,\dots, F_\ell)$ and $k$ is a clopen condition and the marks are independent.
    
    Define $\mathcal{A}(m,k,\ell)$ to be the union of $\mathcal{A}(\sigma,k,F_1,\dots,F_\ell)$ over sequences $(\sigma,k,F_1,\dots,F_\ell)$ that satisfies (1.)--(3.) and similarly define $\mathcal{A}^n(m,k,\ell)$.
    As there are only finitely many such sequences $(\sigma,k,F_1,\dots,F_\ell)$, we have that
    $$\mathbb{P}(\mathcal{A}^n(m,k,\ell))\to \mathbb{P}(\mathcal{A}(m,k,\ell)).$$
    Observe that the events $\mathcal{A}^n(m,k,\ell)$ and $\mathcal{A}^n(m',k',\ell')$ are disjoint whenever $(m,k,\ell)\not=(m',k',\ell')$, and similarly for $\mathcal{A}(m,k,\ell)$.
    
    The proof is finished by noting that $P({\bf X})\upharpoonright B_R(o)=\mathcal{E}$ if and only if ${\bf X}\in \bigsqcup_{m,k,\ell} \mathcal{A}(m,k,\ell)$, hence there is $N>0$ such that
    $$\lim_{n\to\infty}\mathbb{P}\left(\bigsqcup_{-N\le m,k,\ell\le N}\mathcal{A}^n(m,k,\ell)\right)=\mathbb{P}\left(\bigsqcup_{-N\le m,k,\ell\le N}\mathcal{A}(m,k,\ell)\right)>\mathcal{P}^{\mu}(\widehat{\mathcal{E}})-\epsilon$$
    as desired.
\end{proof}

\end{document}